\documentclass[11pt, reqno]{amsart}
\usepackage[foot]{amsaddr}

\usepackage{amsmath,amsfonts,amsthm,amssymb,amscd}
\usepackage{float,graphicx}
\usepackage{hyperref}
\usepackage{upgreek}
\usepackage[letterpaper,top=3.5cm,bottom=3.5cm,left=3.5cm,right=3.5cm,marginparwidth=1.75cm]{geometry}

\newcommand{\cN}{\mathcal{N}}
\newcommand{\cL}{\mathcal{L}}

\newcommand{\bR}{\mathbb{R}}
\newcommand{\bS}{\mathbb{S}}

\newcommand{\rd}{\mathrm{d}}

\newcommand{\sfd}{\mathsf{d}}

\newtheorem{theorem}{Theorem}[section]

\newtheorem{proposition}[theorem]{Proposition}
\newtheorem{assumption}[theorem]{Assumption}

\newtheorem{example}{Example}

\numberwithin{equation}{section}

\begin{document}    
    \title[The Subsampled Poincar\' e Inequality]{Function Approximation via The Subsampled Poincar\' e Inequality}
    \author{Yifan Chen and Thomas Y. Hou}
    \address{Applied and Computational Mathematics, Caltech, 91106}
    \email{yifanc@caltech.edu, hou@cms.caltech.edu}
    
    \date{\today}
    
    \maketitle
    \begin{abstract}
        Function approximation and recovery via some sampled data have long been studied in a wide array of applied mathematics and statistics fields. Analytic tools, such as the Poincar\'e inequality, have been handy for estimating the approximation errors in different scales. The purpose of this paper is to study a generalized Poincar\' e inequality, where the measurement function is of subsampled type, with a small but non-zero lengthscale that will be made precise. Our analysis identifies this inequality as a basic tool for function recovery problems. We discuss and demonstrate the optimality of the inequality concerning the subsampled lengthscale, connecting it to existing results in the literature. In application to function approximation problems, the approximation accuracy using different basis functions and under different regularity assumptions is established by using the subsampled Poincar\'e inequality. We observe that the error bound blows up as the subsampled lengthscale approaches zero, due to the fact that the underlying function is not regular enough to have well-defined pointwise values. A weighted version of the Poincar\' e inequality is proposed to address this problem; its optimality is also discussed. 
    \end{abstract}
    \tableofcontents
    \section{Introduction}
    Approximating a function based on some partial sampled data has been an important topic in applied mathematics, statistics, and the emerging big data science. For a function that is defined in a continuous domain, analytic tools such as the Poincar\'e inequality have been useful in analyzing the approximation errors. Often, depending on the scale that people are looking at, some model parameters may be potentially very small or large. Getting estimates that can capture the dependence on these parameters and remain valid in the small or large limit regime is crucial for understanding the problem. In this paper, we consider a subsampled lengthscale $h$ in the sampled data, and study the approximation in the finite $h$ regime and small limit regime. Several variants of the Poincar\'e inequalities are investigated to achieve this goal, and we will explain their implications for problems of function approximation and recovery in the subsampled data scenario.
    \subsection{Motivation}
    \label{sec: motivation}
    The Poincar\'e inequality, in one of its forms, states that for a bounded, connected and open domain $\Omega \subset \bR^d$ with a Lipschitz boundary, there exists a constant $C(d,p)$, depending on $d$ and $p$ only, such that for every function $u$ in the Sobolev space $W^{1,p}(\Omega)$, it holds
    \[\|u-(u)_{\Omega} \|_{L^p(\Omega)} \leq C(d,p)  \mathrm{diam}(\Omega) \|Du\|_{L^p(\Omega)}\, . \]
    Here, $(u)_{\Omega}$ is the average of $u$ in $\Omega$, i.e. $(u)_{\Omega}=\int_{\Omega} u(x) \,\rd x/ \mu_d(\Omega)$, and $\|\cdot\|_{L^p(\Omega)}$ stands for the $L^p$ norm of a function in $\Omega$. We use  $\mu_d(\Omega)$ to represent the volume of the $d$-dimensional domain $\Omega$, and $\operatorname{diam}(\Omega)$ is the corresponding diameter.
    
    This inequality leads to a nice implication in problems of function approximation and recovery. Consider a function $u$ in $W^{1,p}(\Omega)$ and we know it has bounded oscillation in the sense that $\|Du\|_{L^p(\Omega)} \leq M$ for some $M>0$. To gain more information about $u$, we measure the average data $(u)_{\Omega}$, and try to recover $u$ as accurately as possible using this data. A simple choice of the recovery can be the constant function $(u)_{\Omega}$.  Despite being so simple, guaranteed error control in the $L^p$ norm, according to the Poincar\'e inequality, is given by $ C(d,p) \text{diam}(\Omega)M$, in the worst case.
    
    The data $(u)_{\Omega}$, being an average of $u$ in the whole domain $\Omega$, is of a global scale; it is thus inadequate to capture the fine-scale information. 
    Therefore, to improve the approximation accuracy, a straightforward strategy is to place more sensors in the physical domain $\Omega$, and measure more refined data in small scales. For demonstration of ideas, we assume the domain $\Omega=[0,1]^d$ and it is partitioned evenly into $1/H^d$ cubes, each with lengthscale $H$; see Figure \ref{fig:two-scalegrid}. Mathematically, we write $\Omega= \bigcup_{i\in I} \omega_i^H$ where $\omega_i^H$ is the cube indexed by $i \in I$; we have the cardinality $|I|=1/H^d$ as desired. For each $i$, the small scale data $(u)_{\omega^H_i}$ is measured, which is the average of $u$ in the local domain $\omega^H_i$. As before, we can build a recovery of $u$ using this data; a simple choice is the piecewise constant function $u^H$, with value  $(u)_{\omega^H_i}$ in the cube $\omega^H_i$ for every $i\in I$. Then, the following error control of this recovery holds:
    \begin{align*}
    \|u-u^H\|^p_{L^p(\Omega)}&=\sum_{i \in I} \|u-(u)_{\omega_i^H}\|^p_{L^p(\omega_i^H)} \\ 
    &\leq C(d,p)^p H^p \sum_{i \in I}  \|Du\|_{L^p(\omega_i^H)}^p
    = C(d,p)^p H^p \|Du\|_{L^p(\Omega)}^p\, .
    \end{align*}
    It follows  $\|u-u^H\|_{L^p(\Omega)} \leq C(d,p) MH$. From this estimate, we see that the worst-case error decreases with the rate of $O(H)$ as we refine the measurements. It should be the best error rate in $L^p$ norm that one can expect when we know $\|Du\|_{L^p(\Omega)} \leq M$ only.
    \subsection{Generalization to Subsampled Data}
    The example in the first subsection demonstrates the usefulness of the Poincar\' e inequality for estimating recovery residues. Many estimates in approximation theory and numerical analysis, e.g., the error estimate of finite element methods, rely on similar ideas. Error control in the small scales is established first, and then a suitable global coupling scheme yields the final approximation. Inspecting the example, we observe there may be two potential components that can be further generalized: (1) the type of measurement data, which is an average of the function in the local cube; (2) the local recovery basis function, which is set as a constant in each cube. 
    
        For the first component, in this paper, we are interested in \textit{subsampled data}, which is an average of $u$ in the set $\omega_i^{h,H} \subset \omega_i^H$ that has a possibly smaller lengthscale compared to that of the patch $\omega_i^H$ for each $i \in I$; that means, $h \leq H$, see Figure \ref{fig:two-scalegrid} for an illustration in the two-dimensional case. It is a generalization for the $h=H$ case in Subsection \ref{sec: motivation}. Why shall we consider such a generalization? In physics, the measurement data of a field is often the macroscopic averaged quantity, sometimes represented by the integration over a small region; that is called the frequency/energy truncation, and anything with a smaller lengthscale is ignored. The subsampled measurements match the context naturally. Furthermore, the subsampled data is more general than the data of the Diracs type, which corresponds to the case $h=0$. The pointwise value of a function is not well-defined if the function does not have enough regularity, as a result of the Sobolev embedding theorem \cite{evans_partial_2010}. Thus, studying the behavior of these subsampled data may lead to a more well-behaved yet general mathematical problem.
    Another possibility is setting these small scale data into integration against some low-dimensional sets rather than local cubes, for example, hyperplanes. In that case, the scale of data becomes anisotropic.
    \begin{figure}
        \centering
        \includegraphics[width=0.7\linewidth]{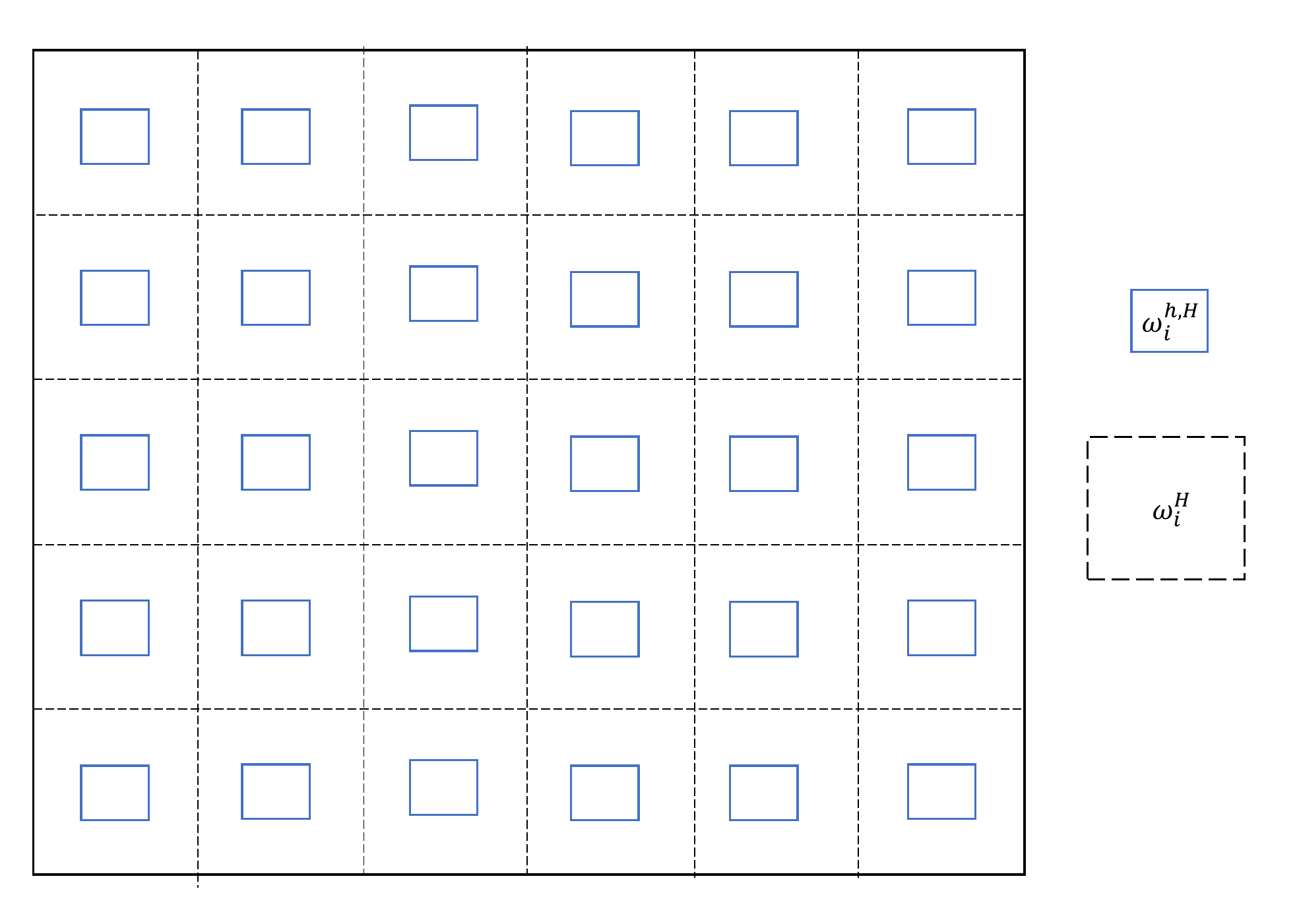}
        \caption{Domain $\Omega=[0,1]^2$; the local cube $\omega_i^{H}$ and the subsampled cube $\omega_i^{h,H}$}
        \label{fig:two-scalegrid}
    \end{figure}
    
    Regarding the second component, i.e., the local recovery basis functions associated with the sampled data, there has been a vast literature discussing the case $h=0$, such as the field of interpolation, approximation theory, spline theory in numerical analysis, Gaussian process regression and kernel regression in non-parametric statistics. Constructing a good recovery with optimal error estimates is essential in these fields. The case $h\neq 0$ relates to Cl\'ement interpolation \cite{clement1975approximation} and has found lots of applications in adaptive finite element methods. The case $h=H$ has been recently connected to applications in numerical homogenization and multiscale computational methods for PDEs \cite{owhadi_multigrid_2017, owhadi2019operator}; see also the $h=0$ case \cite{owhadi2014polyharmonic} for such an application. The $0<h< H$ case in such a setting has not been explored in depth yet. This subsampled case is the one that we would like to investigate in more detail, to understand its implications in contexts of function approximation and multiscale PDEs. This paper concentrates on the problem of function approximation, while we will include discussions for multiscale PDE problems and other applications in our subsequent paper \cite{chen-hou_numerical}.
    
    \subsection{Our Contributions}
    In the first part of this work, we establish a generalized Poincar\'e inequality, discuss the proof strategy, and prove its optimality concerning $h$, in the setting of subsampled data, in Section \ref{sec: general poincare}. Similar result has been obtained in the literarure; we will  discuss them in the corresponding sections.
    We also cover the case when the subsampled data is integration against some low-dimensional hyperplanes. 
    
    Given this subsampled Poincar\'e inequality, we move to study different local basis functions for the recovery that can attain desired approximation accuracy when $u$ is in different function spaces. We start with the piecewise constant recovery, in the same spirit as Subsection \ref{sec: motivation}. To improve the regularity of basis functions, we borrow ideas in the spline approximation theory and establish the corresponding error estimates in Section \ref{sec: improving the basis function}. This approach directly connects to the context of multiscale PDEs; see the work of rough polyharmonic splines \cite{owhadi2014polyharmonic} and Gamblets \cite{owhadi_multigrid_2017}.
    
    In the problem of function approximation, we observe that when the underlying function is not regular enough, the error bound of the recovery blows up as we decrease the subsampled scale $h$. The reason is that the pointwise value of a $W^{1,p}(\Omega)$ function is not well-defined if $d\geq p$. This degeneracy is not desired, and the second part of the work in Section \ref{sec:weighted estimate and semi-supervised} is to discuss a way to fix this issue. We found that if more structures are imposed on $u$, here typically $\int_{\Omega} w(x)|Du(x)|^p\, \rd x < \infty$ for some weight functions that are singular at the data points, then we can obtain non-degenerate recovery. We establish a weighted Poincar\' e inequality to analyze the recovery accuracy; the optimality of this weighted inequality is also discussed.
    
    \subsection{Related Works} We list the related work in terms of different topics. There has been a vast literature on the Poincar\'e inequality of different types, function approximation and recovery, and weighted spaces and inequalities. It is not our goal here to provide an exhaustive review; we will mainly cover papers that we found to exhibit the most direct connection to this work.
    \subsubsection{Generalized Poincar\'e Inequality}
    \label{subsection: generalized poincare}
    Many people have considered extending the constant $(u)_{\Omega}$ in the Poincar\' e inequality to a general linear functional on $u$. In \cite{meyers1978integral,meyers1977integral}, the authors analyzed the condition of the functional in great depth. In Chapter 4 of \cite{ziemer2012weakly}, a unified approach of the Poincar\'e inequality was discussed by studying the norm of this linear functional.  In \cite{alessandrini2008linear}, the linear constraints in Poincar\' e and Korn type inequalities were investigated. Our subsampled measurements can be seen as a special case of their linear functional or linear constraints. Nevertheless, the motivation is different, and their results do not directly lead to the optimal rate on $h$. In the literature, we found a result similar to ours in Corollary 2.7 of \cite{ruiz_note_2012} with a different proof strategy. In the critical $p=d$ case, their rate on $h$ is a little tighter than ours up to a logarithmic term. We show that this rate is indeed optimal concerning $h$ in Proposition \ref{prop: sharpness of the rate}.
    \subsubsection{Function Recovery and Basis Functions}
    The optimal recovery problem has been framed in \cite{micchelli1977survey}. The authors of the book \cite{owhadi2019operator} discuss the game-theoretical and Bayesian methods for optimal recovery and numerical homogenization, which serve as one of the main motivations of this work. Finding appropriate basis functions that can yield a smaller recovery error is essential; often, piecewise constant recovery cannot do the best job, and we need to consider recovery with better regularity. The general strategy we adopt for improving the regularity of basis functions is to apply the inverse of some differential operator, say $\cL=-\nabla\cdot (a\nabla\cdot)$, to the subsampled constant measurement functions supported in the domain $\omega_i^{h,H}$. When $h=0$ and the coefficient $a$ in the elliptic operator is constant one (i.e., $\cL$ is the negative Laplacian operator), our improved basis function reduces to the polyharmonic splines \cite{harder1972interpolation, duchon1977splines}. When $h=0$ or $H$, and the coefficient $a$ is in $L^{\infty}(\Omega)$, then the improved basis reduces to Gamblets \cite{owhadi_multigrid_2017} and rough polyharmonic splines \cite{owhadi2014polyharmonic}. In this paper, we mainly study the improved basis functions for $0<h<H$. We remark that in \cite{owhadi2019operator}, the discussion of the measurement function entails a great generality, and some general conditions on the measurements are proposed to guarantee the approximation accuracy. Our $h \in (0,H)$ case does satisfy their conditions, but their results do not cover the optimal dependence regarding $h$. For the function recovery using subsampled data, obtaining the optimal recovery rate is important, which is the focus of our current work.
    \subsubsection{Weighted Inequality and Degeneracy}
    There has been a vast literature on the weighted Sobolev space and weighted Poincar\' e inequality. To the best of our knowledge, most of them focus on the scenario that both the left-hand side and right-hand side of the inequality are weighted. In our case, we only set the gradient norm on the right-hand side to be weighted. Moreover, we can connect this inequality to applications in function recovery that suffers from degeneracy as $h$ tends to zero. A similar issue in the context of graph Laplacian based semi-supervised learning has bee discussed in \cite{Semi-supervised}. Since then, there has been a lot of literature dealing with this issue, for example, by using higher-order regularization \cite{high_order_regularization}, $p$-Laplacian regularization for large $p$ \cite{p_Laplacian_Jordan, p_Laplacian_analysis}, Lipschitz learning (corresponding to $p=\infty$) \cite{Lipschitz_learning, Lipschitz_learning_analysis}, or changing the weights in the regularization \cite{WNLL}.  Recently, in \cite{calder_properly-weighted_2018}, a singular weight function is proposed to address this problem, which attains a well-defined continuous limit. Our weight function has a form similar to theirs.
    \subsection{Notation}
    We present our notations here. We use $\chi_A(x)$ for the characteristic function of the set $A$. The diameter of a set $\Omega \subset \bR^d$ is denoted by $\text{diam}(\Omega)$. For a function in Euclidean space $\bR^d$ with variable $x$, i.e. $f(x)$, the integration on a measurable set $A$ against the Lebesgue measure will be denoted by $\int_A f(x)\, \rd x$, while the integration with respect to a measure $\lambda$ will be written as $\int_A f(x) \, \rd \lambda(x)$. When there is no ambiguity, the variable name ``$x$'' in the integration may be omitted for simplicity. $L^p(\Omega)$ stands for the space of $p$th power summable functions over $\Omega$ with the corresponding norm $\|\cdot\|_{L^p(\Omega)}$, and $W^{1,p}(\Omega)$ represents the standard Sobolev space on the domain $\Omega$. We use $|\cdot|$ for both the absolute value of a scalar and the modulus of a vector. When we say a set $\Omega$ is a domain, it refers to a connected, open set. The $d$ dimensional Lebesgue measure of $\Omega \subset \bR^d$ (i.e. the volume) is written as $\mu_d(\Omega)$. For $k < d$, we use $\mu_k(\Gamma)$ to represent the $k$ dimensional Hausdorff measure of a $k$ dimensional measurable subset $\Gamma \subset \bR^d$. The $d$-dimensional ball with center $x \in \bR^d$ and radius $r$ is denoted by $B^d(x,r)$.
    
    Throughout the paper, $C(d,p)$ (resp. $C(d)$) stands for a positive generic constant that depends on $d,p$ (resp. $d$) only and may attain different values at different places.
    
    \subsection{Organization of This Paper} We organize the paper as follows. In Section \ref{sec: general poincare}, we discuss a generalized version of the Poincar\' e inequality. As an application, we establish the optimality of the subsampled Poincar\' e inequality. The case of sliced measurement data, i.e., integration against hyperplanes, is also covered here; related optimality issues are discussed. In Section \ref{sec: improving the basis function}, we consider an improvement of the basis functions using ideas from the spline approximation theory, motivated by the work on rough polyharmonic splines \cite{owhadi2014polyharmonic} and Gamblets \cite{owhadi_multigrid_2017}. In Section \ref{sec:weighted estimate and semi-supervised}, we present a weighted Poincar\'e inequality and use it to deal with the degeneracy issue in the recovery. Finally, we conclude the paper in Section \ref{sec: discussion}. In order to demonstrate the main ideas smoothly without overloading the reader too much, some proofs are deferred to the appendix in Section \ref{sec: appendix}. 
    
    \section{A Generalized Poincar\' e Inequality}
    \label{sec: general poincare} 
    In this section, we provide a generalized version of the Poincar\'e inequality, which allows a general linear functional of $u$ beyond $(u)_{\Omega}$. The generalized Poincare inequality has been studied in some previous works, see e.g. \cite{meyers1978integral,meyers1977integral}. Our purpose is to provide a version with quantitative estimates of the approximation error since we are interested in the optimal rate of the approximation error regarding some small scale parameter. 
    
    We begin with reviewing the approaches for proving the Poincar\'e inequality in the literature, and then present our proofs and applications to subsampled data.
    \subsection{A Review of Techniques}\label{sec: review of techniques for poincare}
    The standard way of proving the Poincar\'e inequality is by the argument of contradiction, thanks to the compactness of related function spaces; see Chapter 5.8.1 of \cite{evans_partial_2010}, and Theorem 12.23 in \cite{first_course_sobolev}. This type of argument leads to the existence of the constant only, with no quantitative characterization. To prove the inequality with an explicit constant, we adopt the strategy of expressing the left-hand side of the inequality directly as an integration of the gradient by Newton-Leibniz's rule, and then estimate the contribution of different parts properly. One can arrange these parts using polar coordinates, leading to estimates suitable for star-shaped domains; see page 164 of \cite{trudinger}, Theorem 12.36 in \cite{first_course_sobolev}, and the proof in \cite{veeser2012poincare}. Our approach is to use a change of variables in the integral and estimate the volumes of some related sets. This approach has been adopted in Poincar\'e's elementary proof of the inequality for $p=2$, according to \cite{poincare_slide} (page 8, Poincar\'e's proof by duplication) and \cite{Poincare1890}. We identify an additional step of a weighted H\"older inequality that can sharpen the dependence on $h$, yielding the optimal rate for the case $d\neq p$.
    
    Another more abstract approach for obtaining quantitative constants is to estimate the norm of a related linear functional in a constrained function subspace; see \cite{meyers1978integral,meyers1977integral}, the unified approach in Chapter 4 of \cite{ziemer2012weakly}, and the work \cite{alessandrini2008linear}. As noted in Subsection \ref{subsection: generalized poincare}, this way of proof can lead to a generalized Poincar\'e inequality, with $(u)_{\Omega}$ replaced by some linear functional on $u$. The proof in the paper \cite{ruiz_note_2012} also relies on this idea.
    
    \subsection{The Main Inequality} In this subsection, we present the proofs of our main results, Theorems \ref{thm: general Poincare} and \ref{thm: general poincare for general p}. To begin with, we present the assumption on the domain below.
    \begin{assumption}
        \label{assumption: for general Poincare}
        Let $\Omega \subset \bR^d \ (d \geq 2)$ be a bounded convex domain. The measure $\lambda$ is non-negative with a unit mass on $\overline{\Omega}$.
    \end{assumption}
    We note that a bounded convex set has a Lipschitz boundary; for this reason, we do not need additional assumptions on the regularity of the boundary. The convexity assumption can be relaxed; see remarks after Theorem \ref{thm: general poincare for general p}.
    
    Under this assumption, we begin with a Poincar\'e inequality for the function space $C^{\infty}(\overline{\Omega})$ in Theorem \ref{thm: general Poincare}; the proof is by calculation using simple calculus.. Then, we generalize it to $W^{1,p}(\Omega)$ for $1\leq p<\infty$ in Theorem \ref{thm: general poincare for general p} through a density argument and a special weighted H\"older inequality. 
    \begin{theorem}
        \label{thm: general Poincare}
        Under Assumption \ref{assumption: for general Poincare}, the following inequality holds for every $u \in C^{\infty}(\overline{\Omega}):$
        \begin{align}
        \|u-\int_{\overline{\Omega}} u \,\rd\lambda \|_{L^1(\Omega)}&\leq \mathrm{diam}(\Omega)\int_{\overline{\Omega}} \left(\int_0^1 \frac{1}{t^d} \lambda(\frac{z-t\overline{\Omega}}{1-t}\cap \overline{\Omega})\, \rd t \right) |Du(z)|  \, \rd z \, .
        \end{align}
    \end{theorem} 
    \begin{proof}
    A direct calculation gives
        \begin{equation}
        \label{eqn: Poincare proof 1}
        \begin{aligned}
        &\|u-\int_{\overline{\Omega}} u \,\rd\lambda \|_{L^1(\Omega)}\\
        =&\int_{\overline{\Omega}} \int_{\overline{\Omega}} (u(x)- u(y)) \,\rd \lambda(x)\rd y\\
        \leq& \int_{\overline{\Omega}} \int_{\overline{\Omega}} |u(x)- u(y)| \,\rd\lambda(x)\rd y \, .
        \end{aligned}
        \end{equation} 
        We express the difference $u(x)-u(y)$ through its derivative $Du$ using the Newton-Leibniz rule:
        \begin{align*}
        |u(x)- u(y)|&=\left|\int_0^1 (x-y)\cdot Du((1-t)x+ty) \,\rd t\right|\\
        &\leq \mathrm{diam}(\Omega) \int_0^1 |Du((1-t)x+ty)| \,\rd t\, .
        \end{align*}
        Plugging the above formula into the integral in \eqref{eqn: Poincare proof 1} and using Fubini's theorem, we obtain
        \begin{equation}
        \label{Poincare proof 2}
        \begin{aligned}
        \int_{\overline{\Omega}}  \int_{\overline{\Omega}}|u(x)- u(y)| \, \rd\lambda(x) \rd y 
        \leq \mathrm{diam}(\Omega) \int_0^1\, \rd t \int_{\overline{\Omega}}\int_{\overline{\Omega}} |Du((1-t)x+ty)| \,\rd\lambda(x) \rd y\, .
        \end{aligned}
        \end{equation}
        For any $0\leq t\leq 1$, we have
        \begin{equation}
        \label{eqn: Poincare proof 3}
        \begin{aligned}
        &\int_{\overline{\Omega}}\int_{\overline{\Omega}} |Du((1-t)x+ty)| \,\rd\lambda(x) \rd y\\
        =&\int_{\overline{\Omega}}\, \rd \lambda(x)\int_{\overline{\Omega}} |Du((1-t)x+ty)| \, \rd y\\
        \overset{(a)}{=}&\int_{\overline{\Omega}}\, \rd \lambda(x) \int_{\overline{\Omega}} |Du(z)|\chi_{(1-t)x+t\overline{\Omega}}(z)\frac{1}{t^d}\, \rd z\\
        =&\frac{1}{t^d}\int_{\overline{\Omega}} |Du(z)|\, \rd z \int_{\overline{\Omega}} \chi_{\frac{z-t\overline{\Omega}}{1-t}}(x)\, \rd\lambda(x)\\
        =&\frac{1}{t^d}\int_{\overline{\Omega}} \lambda(\frac{z-t\overline{\Omega}}{1-t}\cap \overline{\Omega})|Du(z)|\, \rd z \, ,
        \end{aligned}
        \end{equation}
        where we have used the change of variables $z=(1-t)x+ty$ in step $(a)$. Since the set $\Omega$ is assumed to be convex, the whole line will lie inside $\Omega$, a fact which is employed in the above calculation. Combining \eqref{Poincare proof 2} and \eqref{eqn: Poincare proof 3} leads to
        \begin{equation}
        \label{eqn: Poincare proof 4}
        \begin{aligned}
        \int_{\overline{\Omega}} \int_{\overline{\Omega}} |u(x)- u(y)| \,\rd\lambda(x)\rd y \leq 
        \mathrm{diam}(\Omega) \int_{\overline{\Omega}} \left(\int_0^1 \frac{1}{t^d} \lambda(\frac{z-t\overline{\Omega}}{1-t}\cap \overline{\Omega})\, \rd t \right) |Du(z)|  \, \rd z \, .
        \end{aligned}
        \end{equation}
        This implies:
        \begin{equation}
        \begin{aligned}
        \label{eqn: Poincare proof 5}
        \|u-\int_{\overline{\Omega}} u \,\rd\lambda \|_{L^1(\Omega)}&\leq \mathrm{diam}(\Omega)\int_{\overline{\Omega}} \left(\int_0^1 \frac{1}{t^d} \lambda(\frac{z-t\overline{\Omega}}{1-t}\cap \overline{\Omega})\, \rd t \right) |Du(z)|  \, \rd z \, .
        \end{aligned}
        \end{equation}
        The proof is completed.
    \end{proof}
    To move further, we assume a condition on the upper bound of the measure; see Assumption \ref{assumption: growth condition for the measure} below. We mention that this assumption will be satisfied for our subsampled measurements, see Propositions \ref{example: subsample Poincare} and \ref{example: sliced Poincare}, so it is suitable for our purpose. We note that in Section \ref{sec:weighted estimate and semi-supervised}, we will make a more refined estimate on $\lambda(\frac{z-\overline{\Omega}}{1-t}\cap \overline{\Omega})$ rather than using the uniform upper bound independent of $z$ in Assumption \ref{assumption: growth condition for the measure}; the refined analysis enables us to get a weighted inequality.
    \begin{assumption}
        \label{assumption: growth condition for the measure} There exists $\alpha(t)$ such that for every $t \in [0,1]$ and $z \in \overline{\Omega}$ it holds that  $\lambda(\frac{z-\overline{\Omega}}{1-t}\cap \overline{\Omega}) \leq \alpha(t)$. 
    \end{assumption}
    With this assumption, the generalized Poincar\'e inequality for $1\leq p <\infty$ is stated in Theorem \ref{thm: general poincare for general p}.
    \begin{theorem}
        \label{thm: general poincare for general p}
        For $1 \leq p < \infty$, under Assumptions \ref{assumption: for general Poincare}, \ref{assumption: growth condition for the measure}, and the additional assumption that $u \to \int_{\overline{\Omega}} u\, \rd \lambda$ is a bounded linear functional on the function space $W^{1,p}(\Omega)$, the following Poincar\'e type inequality holds for any $u \in W^{1,p}(\Omega):$ 
        \begin{equation}
        \begin{aligned}
        \|u-\int_{\overline{\Omega}} u \,\rd\lambda \|_{L^p(\Omega)}\leq \mathrm{diam}(\Omega)\left(\int_0^1 \frac{\alpha(t)^{\frac{1}{p}}}{t^{\frac{d}{p}}}\, \rd t \right)\|Du\|_{L^p(\Omega)}\, . 
        \end{aligned}
        \end{equation}
    \end{theorem}
    \begin{proof}
        The result of the case $p=1$ is a direct combination of Theorem \ref{thm: general Poincare}, Assumption \ref{assumption: growth condition for the measure}, and the fact that $C^{\infty}(\overline{\Omega}) \cap W^{1,1}(\Omega)$ is dense in $W^{1,1}(\Omega)$. Since $u \to \int_{\overline{\Omega}} u\, \rd \lambda$ is a bounded linear functional on $W^{1,1}(\Omega)$, the limiting procedure is well-defined.
        
        For the case $1 <p < \infty$, we only need to consider $u \in C^{\infty}(\overline{\Omega}) \cap W^{1,p}(\Omega)$ because this set is dense in $W^{1,p}(\Omega)$. Using Jensen's inequality, we obtain
        \begin{align*}
        &\|u-\int_{\overline{\Omega}} u \,\rd\lambda \|_{L^p(\Omega)}^p\\
        =&\int_{\overline{\Omega}} \left(\int_{\overline{\Omega}} (u(x)- u(y)) \,\rd \lambda(x)\right)^p\rd y\\
        \leq& \int_{\overline{\Omega}} \int_{\overline{\Omega}} |u(x)- u(y)|^p \,\rd\lambda(x)\rd y \, .
        \end{align*} 
        Similar to the proof of Theorem \ref{thm: general Poincare}, we use the Newton-Leibniz rule to express the term $u(x)-u(y)$:
        \begin{align*}
        |u(x)- u(y)|^p&=|\int_0^1 (x-y)\cdot Du((1-t)x+ty) \,\rd t|^p\\
        &\overset{(b)}{\leq} \text{diam}(\Omega)^p \left(\int_0^1 w(t)^{-\frac{1}{p-1}} \, \rd t\right)^{p-1}  \int_0^1 w(t)|Du((1-t)x+ty)|^p \,\rd t\, .
        \end{align*}
        Here, the step $(b)$ is due to the H\"older inequality, in which we introduce a weight function $w(t) \geq 0$. This weight function $w(t)$ will be determined in the subsequent calculations. We remark that without a correct choice of the weight function, we would not be able to obtain an inequality with a constant that has an optimal scaling property with respect to $h$, as in Proposition \ref{example: subsample Poincare} and Proposition \ref{example: sliced Poincare}.
        
        Then, by the same change of variables as in \eqref{eqn: Poincare proof 3}, we get
        \begin{align*}
        &\int_{\overline{\Omega}}\int_{\overline{\Omega}} |Du((1-t)x+ty)|^p \,\rd\lambda(x) \rd y \\=& \frac{1}{t^d}\int_{\overline{\Omega}} \lambda(\frac{z-t\overline{\Omega}}{1-t}\cap \overline{\Omega})|Du(z)|^p\, \rd z\\
        \leq & \frac{\alpha(t)}{t^d}\int_{\overline{\Omega}} |Du(z)|^p \, \rd z \, .
        \end{align*}
        Following the same argument as that in \eqref{eqn: Poincare proof 4} and \eqref{eqn: Poincare proof 5}, we obtain
        \[\|u-\int_{\overline{\Omega}} u \,\rd\lambda \|_{L^p(\Omega)}^p\leq \text{diam}(\Omega)^p\left(\int_0^1 w(t)^{-\frac{1}{p-1}} \, \rd t\right)^{p-1}\left(\int_0^1 \frac{w(t)\alpha(t)}{t^d}\, \rd t \right) \|Du\|_{L^p(\Omega)}^p\, . \]
        Now, we optimize the choice of the weight function $w(t)$. Let 
        \[w(t)^{-\frac{1}{p-1}}= \frac{w(t)\alpha(t)}{t^d}\, ,\]
        which is the condition for the corresponding H\"older inequality to become an equality. Using this weight function, we obtain
        \[\|u-\int_{\overline{\Omega}} u \,\rd\lambda \|_{L^p(\Omega)}\leq \text{diam}(\Omega)\left(\int_0^1 \frac{\alpha(t)^{\frac{1}{p}}}{t^{\frac{d}{p}}}\, \rd t \right)\|Du\|_{L^p(\Omega)}\, . \]
        This completes the proof.
    \end{proof}
    
    As noted before, some requirements in Assumption \ref{assumption: for general Poincare} can be relaxed, such as the convexity of the domain and also the regularity of the boundary; see several remarks below.
    \begin{enumerate}
        \item The convexity assumption of the domain $\Omega$ can be relaxed. For general non-convex domains, we can use the Sobolev extension theorem to extend the function to a larger convex domain, for example, a ball. More precisely, let $\Omega \subset B^d(0,r)$ for some $r > 0$. We use $u$ to represent both the function in $\Omega$ and its extension to $B^d(0,r)$. We define the extension of the measure $\lambda$ to be zero outside $\overline{\Omega}$, so that $\lambda$ is still a measure with a unit mass for the ball. Then, under the assumptions in Theorem \ref{thm: general poincare for general p}, we have the estimate:
        \begin{align*}
            \|u-\int_{\overline{\Omega}} u \,\rd\lambda \|_{L^p(\Omega)}&\leq \|u-\int_{B^d(0,r)} u \,\rd\lambda \|_{L^p(B^d(0,r))}\\
            &\leq  2r\left(\int_0^1 \frac{\alpha(t)^{\frac{1}{p}}}{t^{\frac{d}{p}}}\, \rd t \right)\|Du\|_{L^p(B^d(0,r))}\, , 
        \end{align*}
        where in the second inequality we use the generalized Poincar\'e inequality for the ball; now $\alpha(t)$ is defined in Assumption \ref{assumption: growth condition for the measure} with $\overline{\Omega}$ replaced by $B^d(0,r)$. Since $\lambda$ is a measure with unit mass, we can assume without loss of generality that $\int_{\overline{\Omega}} u(x) \,\rd x=0$ in the above inequality. Then, by the property of the Sobolev extension theorem, we can further bound
        \begin{align*}
            \|Du\|_{L^p(B^d(0,r))}\lesssim \|u\|_{L^p(\Omega)}+\|Du\|_{L^p(\Omega)} \lesssim \|Du\|_{L^p(\Omega)}\, ,
        \end{align*}
        where we use $x \lesssim y$ to indicate there exists some constant $C$ independent of $u$ such that $x\leq Cy$. The second inequality is due to the assumption $\int_{\overline{\Omega}} u(x)\, \rd x =0$ and the standard Poincar\'e inequality for the domain $\Omega$. Thus, we obtain the generalized Poincar\'e inequality for the domain $\Omega$. Moreover, if $\operatorname{diam}(\Omega)$ is of order $r$, then we can replace $r$ by $\operatorname{diam}(\Omega)$ in the estimate, yielding a similar form as in Theorem \ref{thm: general poincare for general p}.
        In this regard, we only need the assumption of the domain that allows the Sobolev extension theorem to hold.
        \item In Assumption \ref{assumption: for general Poincare}, we require a convex domain, which leads to a Lipschitz boundary. For non-convex domain this property may fail. Nonetheless, when $\lambda$ has no mass in the boundary, we may not need any restrictive assumption on the boundary. The density argument of Meyers-Serrin can apply to any generic domain, i.e., $C^{\infty}(\Omega) \cap W^{1,p}(\Omega)$ is always dense in $W^{1,p}(\Omega)$ and all the arguments in the proof follows smoothly. However, when $\lambda$ has mass on the boundary, we need $C^{\infty}(\overline{\Omega}) \cap W^{1,p}(\Omega)$ to be dense in the proof, which requires the regularity assumption on the boundary.
    \end{enumerate}
    That being said, the present version is enough for our purpose of applications in function recovery and multiscale PDEs with subsampled data; this is the topic of the next subsection.
    
        
    \subsection{Applications} As we have seen, Theorem \ref{thm: general poincare for general p} can be applied to a general measure $\lambda$. In this subsection, we choose this general measure in some special form and obtain several specific Poincar\'e inequalities. 
    \subsubsection{Subsampled Data}First, we choose $\lambda$ to be subsampled in a smaller domain, matching the discussion on the subsampled data before. This leads to the following Proposition \ref{example: subsample Poincare}; its proof is in Subsection \ref{subsec: proof subsampled poincare}.
    \begin{proposition}[Subsampled Poincar\'e inequality]
        \label{example: subsample Poincare}
        Consider a bounded convex domain $\Omega \subset \bR^d$ and its measurable subset $D \subset \Omega$. Let $\mu_d(\Omega)=H^d, \mu_d(D)=h^d$, then for any $1 \leq p < \infty$ and $u \in W^{1,p}(\Omega)$, the following inequality holds:
        \[\|u-\frac{1}{h^{d}}\int_D u \|_{L^p(\Omega)} \leq C(d,p) \mathrm{diam}(\Omega) \tilde{\rho}_{p,d}(\frac{H}{h}) \|Du\|_{L^p(\Omega)} \, ,\]
        where
        \begin{equation*}
        \tilde{\rho}_{p,d}(x) =\left\{
        \begin{aligned}
        &1, &d < p \\
        &\ln (x+1), &d=p  \\
        &x^{\frac{d-p}{p}} &d> p \, .
        \end{aligned}
        \right.
        \end{equation*}
        and $C(d,p)$ is a constant that depends on $d$ and $p$ only.
    \end{proposition}
    In the literature, we found that in Corollary 2.7 of \cite{ruiz_note_2012}, a similar rate on $h$ is obtained through a different approach. Their strategy is to bound the norm of the related linear functional for a constrained function space, as mentioned in Subsection \ref{sec: review of techniques for poincare}. In the critical case $d=p$, the author of \cite{ruiz_note_2012} uses the tool of Orlicz's space to estimate the norm of the functional, which yields a $\log$ dependence on $H/h$. Indeed, their result is a little tighter in the power of $\log$ than ours. Based on their results, the rate function can be improved to 
    \begin{equation}
    \label{eqn: sharp rate}
    \rho_{p,d}(x) =\left\{
    \begin{aligned}
    &1, &d < p \\
    &(\ln (x+1))^{\frac{d-1}{d}}, &d=p  \\
    &x^{\frac{d-p}{p}} &d> p \, .
    \end{aligned}
    \right.
    \end{equation}
    Thus, the improved subsampled Poincar\'e inequality is given by
    \[\|u-\frac{1}{h^{d}}\int_D u \|_{L^p(\Omega)} \leq C(d,p) \mathrm{diam}(\Omega) \rho_{p,d}(\frac{H}{h}) \|Du\|_{L^p(\Omega)} \, .\]
    Now, we demonstrate the optimality of the above rate concerning $h$, in the case when $\Omega, D$ are balls; see the following proposition \ref{prop: sharpness of the rate}. The proof of this proposition is given in Subsection \ref{subsec: sharpness subsampled}. We would like to mention that the choice of domain being balls is to simplify the construction of critical examples. The optimality shall hold for more general domains by following similar ideas. 
    \begin{proposition}[Optimality of the rate]
        \label{prop: sharpness of the rate}
        Let $\Omega=B^d(0,1), D_h=B^d(0,h)$ be the balls centered at $0$ with radius $1$ and $0< h\leq 1/4$ respectively. Then, for $d \geq p$, there exists a constant $C(d,p)$ that depends on $d$ and $p$ only, such that we can find a sequence of functions $u_h \in W^{1,p}(\Omega)$ that satisfy
        \[\frac{\|u_h-\frac{1}{\mu_d(D_h)}\int_{D_h}u_h\|_{L^p(\Omega)}}{\|Du_h\|_{L^p(\Omega)}}\geq C(d,p)\rho_{p,d}(\frac{1}{h})\, , \]
        for any $0< h\leq 1/4$. 
    \end{proposition}    
    Before we move to the second example, let us discuss the implication of the subsampled inequality for function approximation and recovery. Suppose we have the measurement data $\{(u)_{\omega^{h,H}_i}\}_{i \in I}$, then, following the same construction as in the introduction, we get the error bound of the piecewise constant recovery:
        \[ C(d,p)H\rho_{p,d}(\frac{H}{h})\|Du\|_{L^p(\Omega)} \, . \]
        Inspecting this formula, we see that if the ratio $H/h>0$ is fixed, then the error still achieves the $O(H)$ rate for functions in the space $W^{1,p}(\Omega)$. If $p\leq d$, then taking $h \to 0$, the error bound will blow up. This is due to the fact that the Sobolev embedding theorem fails to embed $W^{1,p}(\Omega)$ to the functional space consisting of continuous functions.
    \subsubsection{Sliced Data}    
    \label{sec: sliced data}
    As a second application, we consider the sliced version of the subsampled data and prove the corresponding Poincar\'e inequality, in the following Proposition \ref{example: sliced Poincare}; the proof is deferred to Subsection \ref{subsec: proof sliced poincare}.
    \begin{proposition}[Subsampled Poincar\'e inequality with sliced data]
        \label{example: sliced Poincare}
        Consider a bounded convex domain $\Omega \subset \bR^d$ and a hyperplane $\Gamma \subset \overline{\Omega}$ with dimension $d-1$. Let $\mu_{d-1}(\Gamma)=h^{d-1}$, and suppose that for every hyperplane contained in $\Omega$ that is parallel to $\Gamma$, its $d-1$ dimensional Hausdorff measure is bounded by $H^{d-1}$. Then for any $1<p <\infty$ and $u \in W^{1,p}(\Omega)$, the following inequality holds:
        \[\|u-\frac{1}{h^{d-1}}\int_\Gamma u \|_{L^p(\Omega)} \leq C(d,p) \mathrm{diam}(\Omega) \tilde{\rho}_{p,d}(\frac{H}{h}) \|Du\|_{L^p(\Omega)}\, , \]
        where
        \begin{equation*}
    \tilde{\rho}_{p,d}(x) =\left\{
    \begin{aligned}
    &1, &d < p \\
    &\ln (x+1), &d=p  \\
    &x^{\frac{d-p}{p}} &d> p \, .
    \end{aligned}
    \right.
    \end{equation*}
        and $C(d,p)$ is a constant that depends on $d$ and $p$ only.
    \end{proposition}
    Regarding the optimality of this rate with respect to $h$, we have the following proposition:
    \begin{proposition}[Optimality of the rate, the sliced data case]
    \label{prop: optimal rate demon, sliced case}
        Let $\Omega=B^d(0,1)$ be the ball centered at $0$ with radius $1$, and $D_h$ is a ball in the $(d-1)$-dimensional hyperplane $\{x\in \bR^d: x_d=0\}$, with center $0$ and radius $h$, which satisfies $0< h\leq 1/4$. Then, for $d \geq p$, there exists a constant $C(d,p)$ that depends on $d$ and $p$ only, such that we can find a sequence of functions $u_h \in W^{1,p}(\Omega)$ that satisfy
        \[\frac{\|u_h-\frac{1}{\mu_{d-1}(D_h)}\int_{D_h}u_h\|_{L^p(\Omega)}}{\|Du_h\|_{L^p(\Omega)}}\geq C(d,p)\rho_{p,d}(\frac{1}{h})\, , \]
        for any $0< h\leq 1/4$.
    \end{proposition}    
    By the proposition, our rate is optimal for $d\neq p$ case, while there is still a logarithmic gap in the critical $d=p$ case. It may be possible to improve the rate from $\tilde{\rho}_{p,d}$ to $\rho_{p,d}$ using the technique of Orlicz's space in \cite{ruiz_note_2012}.
    
    We make several remarks for the sliced data case below. 
    \begin{enumerate}
        \item Similar to the subsampled case, if we use the sliced data to make the piecewise constant recovery, the error bound is given by
        $C(d,p)H\tilde{\rho}_{p,d}(H/h)\|Du\|_{L^p(\Omega)}$.
        \item In the sliced data case, we have the measurement functional supported on a hyperplane with co-dimension $1$. One may wonder whether the above proposition can be extended to measurement data that is integration against a set with co-dimension higher than $1$. For that case, we can still use Theorem \ref{thm: general poincare for general p} since it works for general measurement functional $\int_{\overline{\Omega}} u \, \rd \lambda$. Following similar calculations as those in the proof of Propositions \ref{example: subsample Poincare} and \ref{example: sliced Poincare}, it is possible to get the corresponding Poincar\'e inequality. Nevertheless, the optimality of the rate may require more delicate discussions, especially for the critical case.
        
        Moreover, to have the Poincar\'e inequality, we need the assumption that $u \to \int_{\overline{\Omega}} u\, \rd \lambda$ is a bounded linear functional on the function space $W^{1,p}(\Omega)$. Thus, we may not allow the measurement data to be an integration against a set of very low dimensions if $p$ is not large enough. According to the trace theorem, the co-dimension $m$ of the set should satisfy $1-m/p>0$ here, so the dimension of the set needs to be strictly larger than $d-p$.
    \end{enumerate}
    Overall, the two applications in this subsection demonstrate the usefulness of Theorem \ref{thm: general poincare for general p}. It is of future interest to find more applications where Theorem \ref{thm: general poincare for general p} can lead to optimal scaling rate concerning some parameters of interest and to improve the inequality in the critical case for the rate regarding the small scale parameter $h$.

    \section{Improved Multiscale Basis Functions} 
    In the last section, we have discussed the subsampled Poincar\'e inequality and its implication to problems of function approximation and recovery. The discussion is mainly focused on how to establish the related inequality; its application to function recovery is limited to constant basis functions. In this section, we consider an improvement on the regularity of the basis functions for the case $p=2$. The generalization to $p \neq 2$ is left for future research.
    \label{sec: improving the basis function}
    \subsection{Construction of Basis Functions} Our strategy is to borrow ideas in variational splines and the recent progress in numerical homogenization for constructing multiscale basis functions \cite{owhadi2019operator}. We begin with some definitions that will become useful.
    \subsubsection{Domains, Operators and Norms}
    As before, we consider $\Omega=[0,1]^d$, and its decomposition into cubes follows the same setting; the reader can look at Figure \ref{fig:two-scalegrid} for the setup of the problem. 
    
    We introduce the notation $\cL=-\nabla \cdot (a \nabla \cdot)$; it is an elliptic operator with homogeneous Dirichlet boundary condition. The coefficient $a:\Omega \to \bR$ is assumed to satisfy $0 < a_{\text{min}} \leq a(x) \leq a_{\text{max}}<\infty$ for all $x \in \Omega$.
    
     Given the coefficient function $a$, we define the associated energy norm for any $u \in H_0^1(\Omega)$  by 
     \[\|u\|^2_{H_a^1(\Omega)}:=\int_\Omega a(x)|\nabla u(x)|^2\, \rd x\, ;\] furthermore, the induced inner product is denoted by $\left<\cdot,\cdot\right>_a$ such that for $u,v \in H_0^1(\Omega)$, it holds
    \[\left<u,v\right>_a=\int_\Omega a(x)\nabla u(x) \cdot \nabla v(x) \, \rd x\, .\]
    Recall in our introduction, we assume the function to be recovered satisfies $\|Du\|_{L^p(\Omega)}\leq M$. In this section, we will assume $\|u\|_{H_a^1(\Omega)}<\infty$, or additionally, $\|\cL u\|_{L^2(\Omega)}<\infty$, to study how the regularity of the basis functions can influence the accuracy of the recovery, given different assumptions on $u$.
    \subsubsection{Measurement Functions and Basis Functions} First, we introduce a notation for describing the subsampled data. We write the subsampled measurement functions by $\{\phi^{h,H}_i \}_{i \in I}$ where each $\phi^{h,H}_i$ is an indicator function of the patch $\omega_i^{h,H}$, normalized to have unit $L^1$ norm. Under this context, we can write the subsampled data \[(u)_{\omega_i^{h,H}}=[u,\phi^{h,H}_i]\, ,\] where $[\cdot,\cdot]$ is the $L^2$ inner product. Therefore, the problem becomes recovering $u$ from the data $[u,\phi^{h,H}_i], 1\leq i \leq I$. The piecewise constant recovery can be writted in the following form:
    \[u^{\text{pc}}=\sum_{i \in I} [u,\phi_{i}^{h,H}]\varphi_{i}^{h,H}\, , \]
    where each $\varphi_{i}^{h,H}$ is the basis function being constant $1$ supported in $\omega_i^{H}$.
    
    Now, we consider the improved multiscale basis functions, denoted by $\{\psi^{h,H}_i\}_{i\in I}$, which solve the following optimization problem:
    \begin{equation}
    \label{eqn: optimization def basis}
    \begin{aligned}
    \psi_{i}^{h,H} = \text{argmin}_{\psi \in H_0^1(\Omega)}\quad  &\|\psi\|_{H_a^1(\Omega)}^2 \\
     \text{subject to}\quad &[\psi, \phi_j^{h,H}] = \delta_{i,j}\ \  \text{for}\ \  j \in I \, ,
 \end{aligned}
    \end{equation}
    where $\delta_{i,j}=1$ if $i=j$, and has value $0$ otherwise. We use the term ``multiscale basis'' here because the energy norm is involved in the optimization; the multiscale behavior of $a$ is transfered to the basis functions. If $a$ is oscillatory, then $\psi_i^{h,H}$ will have a similar oscillation.
    
    With these basis functions, the recovered function is constructed by
    \[u^{h,H}=\sum_{i \in I} [u,\phi_{i}^{h,H}]\psi_{i}^{h,H}\, . \]
    According to \cite{owhadi2019operator}, this recovery is minimax optimal in the relative energy norm, given the data $[u,\phi_i^{h,H}]$ for $i \in I$. Through a Bayesian perspective, it can also be understood as the mean of the Gaussian process $\xi\sim \cN(0, \cL^{-1})$, conditioned on the observation data $[\xi,\phi_i^{h,H}]=[u,\phi_i^{h,H}]$ for $i \in I$.
    \subsubsection{Property of The Improved Basis Functions}
    We mention two properties of the improved basis function, which would be helpful for understanding its implications. For more discussions we refer to the book \cite{owhadi2019operator}. 
    
    The first property is the relation $\text{span}_{i \in I}~\{\psi^{h,H}_i\}=\text{span}_{i \in I}~\{\cL^{-1}\phi^{h,H}_i\}$, so that $\psi_i^{h,H}$ is given by a linear combination of $\cL^{-1} \phi^{h,H}_j$ for $j \in I$; see the following Proposition \ref{prop: explicit form of spline}. In this sense, the regularity of basis functions is improved by applying the inverse of an differential operator, here being $\cL$, to the measurement functions $\{\phi^{h,H}_i\}_{i \in I}$.
    
    We use the notation that $|I|$ is the cardinality of the index set $I$.
    \begin{proposition}
        \label{prop: explicit form of spline}
        For each $i \in I$, the basis function $\psi_i^{h,H}$ has the form
        \[\psi_i^{h,H}=\sum_{j\in I} \Theta_{i,j}^{-1}\cL^{-1}\phi_j^{h,H} \, , \]
        where $\Theta \in \bR^{|I|\times |I|}$ with entries $\Theta_{i,j}=[\phi_j^{h,H},\cL^{-1}\phi_i^{h,H}]$ and $\Theta^{-1}$ is the inverse of $\Theta$.
    \end{proposition}
    The proof of the above proposition follows the same strategy as that of proving Theorem 3.1 in \cite{owhadi_multigrid_2017}; one can also easily understand the result by using Lagrange multipliers for the constrained optimization problem \eqref{eqn: optimization def basis}.
    
    The second property is about the Galerkin orthogonality of the recovered function $u^{h,H}$; see Proposition \ref{prop: u h, H is projection}.
    \begin{proposition}
        \label{prop: u h, H is projection}
        The function $u^{h,H}$ is the projection of $u$ into the function space spanned by $\{\psi^{h,H}_i \}_{i \in I}$, under the inner product $\left<\cdot,\cdot\right>_a$.
    \end{proposition}
    \begin{proof}
        It suffices to show $u-u^{h,H}$ is orthogonal to $\psi_i^{h,H}$ for any $i \in I$ under the inner product $\left<\cdot,\cdot\right>_a$. Equivalently, we need to show \[\left<u-u^{h,H},\psi_i^{h,H}\right>_a=0\, .\] Since $\psi^{h,H}_i \in \text{span}_{i \in I}~\{\cL^{-1} \phi^{h,H}_i \}$, this is equivalent to $[u-u^{h,H},\phi^{h,H}_i]=0$. Observing that
        \[ [u-u^{h,H},\phi^{h,H}_i]=[u,\phi^{h,H}_i]-\sum_{j \in I} [u,\phi^{h,H}_j][\phi^{h,H}_i,\psi^{h,H}_j]=0 \, ,  \]
        by the definition of $u^{h,H}$ and $\psi_i^{h,H}$, we complete the proof.
    \end{proof}
    With these two useful properties, we move to study the accuracy of the recovery $u^{h,H}$ in the next section.
    \subsection{Error Estimates Adapted to Regularity}
    In this section, we derive the approximation accuracy of the above recovery. We discuss two assumptions on $u$: (1) $u \in H_0^1(\Omega)$, which corresponds to the setting in the piecewise constant recovery before, i.e., we have the bounded norm for the gradient; (2) we further have the information $\cL u \in L^2(\Omega)$; this is an improved regularity assumption on $u$. We can readily see the improvement if we set $a$ to be a constant function with value $1$, in which case $\cL$ becomes the negative Laplacian operator. Then, $\cL u \in L^2(\Omega)$ implies $u \in H^2(\Omega)$, an improved regularity for $u$. 
    
    We encompass the discussion of general $\cL$ here, as it is of interest in multiscale elliptic PDEs, where the conductivity field $a$ can exhibit strong heterogeneity. In the following, Theorem \ref{thm:approximation error, subsampled spline} shows the error estimate adapted to the regularity of $u$; its proof relies on the subsampled Poincar\'e inequality that we have established in Section \ref{sec: general poincare}. 
    \begin{theorem}
        \label{thm:approximation error, subsampled spline}
        Under the assumption that $u \in H_0^1(\Omega)$, we have the following error estimate:
        \begin{align*}
        & \|u-u^{h,H}\|_{H^1_a(\Omega)} \leq \|u\|_{H_a^1(\Omega)} \, ,\\
        & \|u-u^{h,H} \|_{L^2(\Omega)} \leq \frac{1}{\sqrt{a_{\min}}}C(d)H\rho_{2,d}(\frac{H}{h})\|u\|_{H_a^1(\Omega)}\, ,
        \end{align*}
        where $C(d)$ is a constant that depends on $d$ only.
        
        Furthermore, under the additional assumption that $\cL u \in L^2(\Omega)$, we have the improved  $H_a^1(\Omega)$ estimate:
        \[ \|u-u^{h,H}\|_{H^1_a(\Omega)} \leq \frac{1}{\sqrt{a_{\min}}}C(d)H\rho_{2,d}(\frac{H}{h})\|\cL u\|_{L^2(\Omega)} \, , \]
        and the improved $L^2(\Omega)$ estimate:
        \[\|u-u^{h,H} \|_{L^2(\Omega)} \leq \frac{1}{a_{\min}}C(d)^2H^2\rho_{2,d}(\frac{H}{h})^2\|\cL u\|_{L^2(\Omega)}\, .\]
    \end{theorem}
    \begin{proof}
        We start the analysis for the case $u \in H_0^1(\Omega)$. The first estimate is readily true by using the property that $u^{h,H}$ is the projection of $u$ under the energy norm $H_a^1(\Omega)$. For the second inequality, we introduce a function $w$ such that $\cL w=u-u^{h,H}$. Then, it follows that
        \begin{equation}
        \label{eqn: v=u-u h H}
        \|u-u^{h,H}\|_{L^2(\Omega)}^2=[u-u^{h,H},u-u^{h,H}]=\left<u-u^{h,H},w\right>_a \, .
        \end{equation}
        Since $u-u^{h,H}$ is orthogonal to every $\psi^{h,H}_i$ under the inner product $\left<\cdot,\cdot\right>_a$ by Proposition \ref{prop: u h, H is projection}, we have
        \begin{equation}
        \label{eqn:proof of spline estimate 1}
        \begin{aligned}
        \left<u-u^{h,H},w\right>_a&=\left<u-u^{h,H},w-\sum_{i \in I} [w,\phi^{h,H}_i]\psi^{h,H}_i\right>_a \\ &\leq \|u-u^{h,H}\|_{H^1_a(\Omega)}\|w-\sum_{i \in I} [w,\phi^{h,H}_i]\psi^{h,H}_i\|_{H^1_a(\Omega)}\, .
        \end{aligned}
        \end{equation}
        Now, we estimate the second term in the above right-hand side. We can write $w^{h,H}=\sum_{i \in I} [w,\phi^{h,H}_i]\psi^{h,H}_i$. From the orthogonality of recovery (Proposition \ref{prop: u h, H is projection}), we get
        \begin{equation}
        \label{eqn: w variational form}
        \|w-w^{h,H}\|_{H^1_a(\Omega)} =\min_{ \{c_i\}_{i\in I}} \|w-\sum_{i \in I} c_i \cL^{-1} \phi^{h,H}_i \|_{H^1_a(\Omega)}\, . 
        \end{equation}
        Therefore, choosing specific $c_i$ yields an upper bound on this term. For ease of notation we write $v=u-u^{h,H}$, and here we choose \[c_i=\int_{\omega_i^H} v, \quad i\in I\, .\] For this choice, let $w_0=\sum_{i \in I} c_i \cL^{-1} \phi^{h,H}_i$, then we get
        \begin{equation}
        \label{eqn: w error estimate}
        \begin{aligned}
        \|w-w_0 \|_{H^1_a(\Omega)}^2 =& \int_{\Omega} (w-w_0)\cL(w-w_0) \\
        =&\int_{\Omega} (w-w_0)(v-\sum_{i \in I} c_i \phi^{h,H}_i)\\
        =&\sum_{i\in I} \int_{\omega_i^H} (w-w_0)(v- c_i \phi^{h,H}_i)\\
        =&\sum_{i \in I} \int_{\omega_i^H} \left(w-w_0-\int_{\omega_i^H} (w-w_0)\phi_i^{h,H}\right)v
        \end{aligned}
        \end{equation}
        where in the last equality we have substituted the formula of $c_i$ into the equation. Then, invoking the subsampled Poincar\'e inequality (recall that $\phi_i^{h,H}$ has unit $L^1$ norm), we get
        \begin{align*}
            \|w-w_0 \|_{H^1_a(\Omega)}^2\leq &\sum_{i \in I} C(d)H\rho_{2,d}(\frac{H}{h})\|D(w-w_0)\|_{L^2(\omega_i^H)}\|v\|_{L^2(\omega_i^H)}\\
        \leq & \frac{1}{\sqrt{a_{\text{min}}}}C(d)H\rho_{2,d}(\frac{H}{h})\|w-w_0\|_{H_a^1(\Omega)}\|v\|_{L^2(\Omega)}\, .
        \end{align*}
        Finally, by using \eqref{eqn: w variational form} and the above estimate, it yields that
        \begin{equation}
        \label{eqn:proof of spline estimate 2}
        \|w-w^{h,H}\|_{H^1_a(\Omega)} \leq \|w-w_0\|_{H^1_a(\Omega)}\leq C(d)\frac{1}{\sqrt{a_{\text{min}}}} H\rho_{2,d}(\frac{H}{h})\|v\|_{L^2(\Omega)}\, .
        \end{equation}
        Further, we obtain by using \eqref{eqn:proof of spline estimate 1} that
        \[\left<u-u^{h,H},w\right>_a\leq \frac{1}{\sqrt{a_{\text{min}}}}C(d) H\rho_{2,d}(\frac{H}{h})\|v\|_{H_a^1(\Omega)}\|v\|_{L^2(\Omega)} \, . \]
        Combining the above estimate with \eqref{eqn: v=u-u h H}, we have
        \begin{equation}
        \label{eqn: H1 to L2 estimate}
        \begin{aligned}
        \|u-u^{h,H} \|_{L^2(\Omega)} &\leq \frac{1}{\sqrt{a_{\min}}}C(d)H\rho_{2,d}(\frac{H}{h})\|u-u^{h,H}\|_{H_a^1(\Omega)} \\ &\leq  \frac{1}{\sqrt{a_{\min}}}C(d)H\rho_{2,d}(\frac{H}{h})\|u\|_{H_a^1(\Omega)} \, .
        \end{aligned}
        \end{equation}
        Thus, we complete the proof for the first part.
        
        For the case $\cL u \in L^2(\Omega)$, we follow the same strategy as outlined in  \eqref{eqn: w variational form}, \eqref{eqn: w error estimate} and \eqref{eqn:proof of spline estimate 2} (apply all the operations on $w$ to the function $u$ and note that $\cL v=w$ in \eqref{eqn:proof of spline estimate 2}), which implies
        \[ \|u-u^{h,H}\|_{H^1_a(\Omega)} \leq \frac{1}{\sqrt{a_{\min}}}C(d)H\rho_{2,d}(\frac{H}{h})\|\cL u\|_{L^2(\Omega)} \, . \]
        So we have obtained the improved estimate in the energy norm. 
        To get the improved $L^2$ estimate, we apply the argument in \eqref{eqn: H1 to L2 estimate}, which leads to 
        \begin{equation*}
        \begin{aligned}
        \|u-u^{h,H} \|_{L^2(\Omega)} &\leq \frac{1}{\sqrt{a_{\min}}}C(d)H\rho_{2,d}(\frac{H}{h})\|u-u^{h,H}\|_{H_a^1(\Omega)}\\ &\leq \frac{1}{a_{\min}}C(d)^2H^2\rho_{2,d}(\frac{H}{h})^2\|\cL u\|_{L^2(\Omega)}\, .
        \end{aligned}
        \end{equation*}
        The proof is completed.
    \end{proof}
    Let us discuss the implication of Theorem \ref{thm:approximation error, subsampled spline}. It shows that under the assumption $\|u\|_{H^1_a(\Omega)} \leq M$, the recovery using piecewise constant functions and using the improved basis functions achieve the same $L^2$-norm accuracy; they are both of order $O(H)$, if the ratio $H/h$ is fixed as we refine $H$. Using the improved basis functions yields a bounded error in the energy norm, i.e., the recovery is stable with respect to the energy norm, as a consequence of Proposition \ref{prop: u h, H is projection}; this property does not hold for the piecewise constant recovery.
    
    Furthermore, when we know additional information that $\|\cL u\|_{L^2(\Omega)}$ is finite, the accuracy of the recovery using the multiscale basis functions is improved, from $O(1)$ to $O(H)$ in the energy norm, and $O(H)$ to $O(H^2)$ in the $L^2$ norm. This phenomenon implies the importance of adapting the regularity of the basis functions to the regularity of the ground truth. 
    
  In addition, we provide several remarks below:
    \begin{enumerate}
        \item Despite the desired property of the improved basis function, its construction requires more computational efforts. The optimization problem is on the global domain $\Omega$. In practical computation, one needs to localize the domain. This difficulty is addressed by observing that $\psi_i^{h,H}$ exhibits exponential decay in the energy norm \cite{malqvist_localization_2014, owhadi_multigrid_2017}, with respect to the distance from the center of the corresponding measurement function $\phi_i^{h,H}$. Thus, the computation can be localized by replacing the global domain $\Omega$ in the constraint set for $\psi \in H_0^1(\Omega)$ of \eqref{eqn: optimization def basis} by some localized oversampling domain around $\omega_i^H$.  Discussions on this issue will be included in our companion paper \cite{chen-hou_numerical} for the computation of multiscale PDEs.
        \item The results in this section also apply to the subsampled measurements with the sliced data type. As we see, the main technique used in the proof is the subsampled Poincar\'e inequality. By Proposition \ref{example: sliced Poincare}, the inequality holds for the sliced data case.
        \item It is possible to obtain basis functions with even higher regularity. We refer to \cite{owhadi2014polyharmonic, owhadi_multigrid_2017, owhadi2019operator, hou_pengchuan}; they mainly focus on the case $h=0$ or $h=H$. The adaptation to the setting $0<h<H$ will be natural.
    \end{enumerate}
    Overall, given the subsampled data, it is important to use appropriate basis functions for the recovery. The generalization to higher-order differential operators or PDEs may need the tool of ``subsampled'' Bramble-Hilbert lemma with an optimal rate on the small scale parameter $h$.
    
    \section{Degeneracy and Weighted Estimate}
    \label{sec:weighted estimate and semi-supervised}
    In the last two sections, we have used the subsampled data to build a recovery of $u$, using piecewise constants or improved basis functions, respectively. From the error estimate, we observed that when $d \geq p$, the error blows up when $h$ goes to $0$. As we mentioned earlier, this phenomenon is not avoidable in general, if we only know that $u$ belongs to $W^{1,p}(\Omega)$. Pointwise evaluations are not stable for functions in this space if $d \geq p$. 
    
    In practice, we often encounter recovery problems in a high dimension. It is natural to ask whether this degeneracy issue can be fixed by imposing more structures on $u$. There has been some work in which $u$ is assumed to be in $W^{k,2}(\Omega)$ for some $k>1$ \cite{high_order_regularization}; this assumption ensures the continuity of the function. Alternatively, one can increase $p$, and when $p > d$, the degeneracy issue disappears; see \cite{p_Laplacian_Jordan, p_Laplacian_analysis, Lipschitz_learning, Lipschitz_learning_analysis}.
    
    In this section, we consider the approach of imposing a singular weight in the gradient norm to tackle the degeneracy issue, motivated by the works \cite{WNLL, calder_properly-weighted_2018}. We study a weighted Poincar\'e inequality as a tool to analyze the recovery error for functions that belong to a weighted space.
    \subsection{A Weighted Poincar\'e Inequality}
    We consider a general $p$ that may not equal $2$, and we assume $d\geq p$; thus, the space $W^{1,p}(\Omega)$ does not embed into the functional space consisting of continuous functions. We start with definitions on the weighted norms and domains.
    \subsubsection{Norms and Domains}
    For a weight function $w>0$, the weighted norm $\|\cdot\|_{L^p_w(\Omega)}$ is defined by 
    \[\|u\|_{L^p_w(\Omega)}:=\left(\int_{\Omega} w(x)|u(x)|^p \, \rd x\right)^{1/p}\, .\] The distance of $x$ to a set $D$ is denoted by $\sfd(x,D)$, and the distance between two sets $A$ and $B$ in Euclidean space is denoted by $\sfd(A,B)$. The domains $D,\Omega$ under consideration satisfy the following assumption.
    \begin{assumption}
        \label{assumption: regular domain}
        There exist positive constants $C_1(d,p)$ and $C_2(d,p)$, such that for the domain $M=\Omega$ or $D$, it holds \[C_1(d,p)^d\mathrm{diam}(M)^d \leq \mu_d(M) \leq C_2(d,p)^d\mathrm{diam}(M)^d\, . \]
    \end{assumption}
    The assumption simply says that the domain cannot deviate too far from a ball. The diameter is a good measure of its shape.
    \subsubsection{The Weighted Inequality} In this subsection, we present the weighted inequalities in Theorems \ref{thm: weight Poincare W11} and \ref{thm: weighted Poincare}; their proofs can be found in Subsections \ref{Proof of Theorem thm: weight Poincare W11} and \ref{Proof of Theorem thm: weighted Poincare}. The proof is an application of Theorem \ref{thm: general Poincare}; here the difference between these two theorems and Theorem \ref{thm: general poincare for general p} is that we characterizes the function $\lambda(\frac{z-\overline{\Omega}}{1-t}\cap \overline{\Omega})$ in Assumption \ref{assumption: growth condition for the measure} in a more refined way than the uniform bound used in Theorem \ref{thm: general poincare for general p}.
    
    We use $\max\{a,b\}$ to represent the maximum of the real numbers $a$ and $b$.
    \begin{theorem}
        \label{thm: weight Poincare W11}
        Let $D \subset \Omega$ satisfy Assumptions \ref{assumption: for general Poincare} and \ref{assumption: regular domain}, with $\mu_d(\Omega)=H^d$ and $\mu_d(D)=h^d$. For every $u \in W^{1,1}(\Omega)$, the following inequality holds:
        \[\|u-\frac{1}{h^{d}}\int_D u \|_{L^1(\Omega)} \leq C(d,p)H \|Du\|_{L^1_w(\Omega)}\, , \]
        where the weight function is chosen to be
        \begin{equation*}
        w(x)=\left(\frac{H}{\max \{h,\sfd(x,D)\}}\right)^{d-1}\, ,
        \end{equation*}
        and $C(d,p)$ is a constant that depends on $d$ and $p$ only.
    \end{theorem}
    The result in Theorem \ref{thm: weight Poincare W11} is a little stronger than Proposition \ref{example: subsample Poincare} for $p=1$. We can easily see this using the fact that $w(x)\leq (H/h)^{d-1}=\rho_{1,d}(H/h)$ for any $x \in \Omega$. From this perspective, the weighted inequality uses more refined spatial information on the gradient norm, compared to the previous subsampled Poincar\'e inequality. 
    
    The weighted inequality for the $p>1$ case is stated below.
    \begin{theorem}
        \label{thm: weighted Poincare}
        Let $D \subset \Omega$ satisfy Assumptions \ref{assumption: for general Poincare} and \ref{assumption: regular domain}, with $\mu_d(\Omega)=H^d$ and $\mu_d(D)=h^d$. For every $u \in W^{1,p}(\Omega)$ with $p > 1$, the following inequality holds true
        \begin{equation}
        \label{eqn: weighted poincare}
           \|u-\frac{1}{h^{d}}\int_D u \|_{L^p(\Omega)} \leq C(d,p)H \|Du\|_{L^p_w(\Omega)}\, , 
        \end{equation}
        if the weight function satisfies the condition
        \begin{equation}
        \label{eqn: requirement on the weight}
        \int_{\Omega} \left(\frac{H}{\max \{h,\sfd(x,D)\}}\right)^{\frac{p(d-1)}{p-1}}w(x)^{-\frac{1}{p-1}}\, \rd x\leq C_w(d,p) H^d \, ,
        \end{equation}
        where $C(d,p)$ and $C_w(d,p)$ are constants that depend on $d$ and $p$ only.
    \end{theorem}
    The above theorem contains a general requirement on the weight function $w$. We will discuss the choices in detail in the next subsection.
        \subsubsection{Examples of The Weights} In this subsection, we present some examples that satisfy the condition \eqref{eqn: requirement on the weight}. We assume $0 \in D \subset \Omega$; otherwise, we can shift the domain.
        
        We begin with weight functions of a polynomial profile in Example \ref{example: weight 1}.
    \begin{example}
        \label{example: weight 1}
        The weight function 
        $$w(x)=\left(\frac{H}{\max \{h,\sfd(x,D)\}}\right)^{d-p+\beta}$$
        for any $\beta > 0$ satisfies the condition in Theorem \ref{thm: weighted Poincare}.
    \end{example}
     \begin{proof}
        We assume $d-\frac{\beta}{p-1}>0$ first. Direct calculation leads to
        \begin{align*}
            &\left(\frac{H}{\max \{h,\sfd(x,D)\}}\right)^{\frac{p(d-1)}{p-1}} w(x)^{-\frac{1}{p-1}}\\
            &=\left(\frac{H}{\max \{h,\sfd(x,D)\}}\right)^{d-\frac{\beta}{p-1}}\, \leq C(d,p)\left(\frac{H}{|x|}\right)^{d-\frac{\beta}{p-1}}\, ,
        \end{align*}
       where we have used the fact that $|x|\leq C(d,p)\max\{h,d(x,D)\}$. We can prove this fact as follows. Assumption \ref{assumption: regular domain} implies that there is a constant $C(d,p)$ such that $\sfd(x,D)\geq |x|-C(d,p)h$ since the diameter of $D$ is bounded by a factor of $h$ and $0 \in D$. Then, it follows that 
       \begin{equation}
       \max\{h,\sfd(x,D)\}\geq \frac{C(d,p)h+\sfd(x,D)}{C(d,p)+1}\geq \frac{1}{C(d,p)+1}|x|\, ,
       \end{equation}
       where we have used the fact that the maximum of two numbers is larger than convex combination of them. 
        Note that we use $C(d,p)$ to represent a generic constant dependent on $d,p$, and its value can vary from place to place.
        
        When $\beta>0$, $x\to |x|^{-d+\frac{\beta}{p-1}}$ is integrable around the origin in $d$ dimensional space. Thus,
        \begin{align*}
            \int_{\Omega} \left(\frac{H}{|x|}\right)^{d-\frac{\beta}{p-1}}\, \rd x \leq \int_{B^d(0,C(d,p)H)} \left(\frac{H}{|x|}\right)^{d-\frac{\beta}{p-1}}\, \rd x \leq C(d,p)\beta H^d\, ,
        \end{align*}
        where $C(d,p)$ is a constant that depends on $d,p$, and can vary its value from place to place. The first inequality is by Assumption \ref{assumption: regular domain}, and the second inequality is by direct integration.
        
        If $d \leq \frac{\beta}{p-1}$, we have the relation: $\max\{h,\sfd(x,D)\}\leq \max\{h,|x|\}$, which implies
        \begin{align*}
            \int_{\Omega} \left(\frac{H}{\max \{h,\sfd(x,D)\}}\right)^{d-\frac{\beta}{p-1}}\, \rd x&\leq \int_{B^d(0,h)}\left(\frac{h}{H}\right)^{\frac{\beta}{p-1}-d}\, \rd x+\int_{\Omega} \left(\frac{|x|}{H}\right)^{\frac{\beta}{p-1}-d}\, \rd x\\
            &\leq C(d,p,\beta)H^d\, ,
        \end{align*}
        for some constant $C(d,p,\beta)$ that depends on $d,p$ and $\beta$ only.
    \end{proof}
    When $\beta=0$ in the above example, we can also supplement some logarithmic correction to ensure the integrability condition; see Example \ref{example: weight 2}.
    \begin{example}
        \label{example: weight 2}
        The weight function 
        $$w(x)=\left(\frac{H}{\max \{h,\sfd(x,D)\}}\right)^{d-p}\left(\log(\frac{1}{\max \{h,\sfd(x,D)\}})+1\right)^{\gamma}/\left(\log(\frac{1}{H})+1\right)^{\gamma-p+1}$$ for $\gamma > p-1$
        satisfies the condition in Theorem \ref{thm: weighted Poincare}.
    \end{example}
    \begin{proof}
            Similarly, we get
            \begin{align*}
                &\left(\frac{H}{\max \{h,\sfd(x,D)\}}\right)^{\frac{p(d-1)}{p-1}}w(x)^{-\frac{1}{p-1}}\\
                =&H^d\left(\frac{1}{\max \{h,\sfd(x,D)\}}\right)^{d}\left(\log(\frac{1}{\max \{h,\sfd(x,D)\}})+1 \right)^{-\frac{\gamma}{p-1}}/\left(\log(\frac{1}{H})+1\right)^{1-\frac{\gamma}{p-1}} \\
                \leq &C(d,p)H^d\left(\frac{1}{|x|}\right)^{d}\left(\log(\frac{1}{|x|})+1 \right)^{-\frac{\gamma}{p-1}}/\left(\log(\frac{1}{H})+1\right)^{1-\frac{\gamma}{p-1}}\, .
            \end{align*}
            The proof is completed by noticing the fact that in $d$ dimension, the function $x\to |x|^{-d}\left(\log(\frac{1}{|x|}+1)\right)^{-\frac{\gamma}{p-1}}$ is integrable around the origin if $\gamma > p-1$.
    \end{proof}
    Similar to the discussion after Theorem \ref{thm: weight Poincare W11}, we can compare these weighted inequalities with the subsampled Poincar\'e inequality before, in the case $d\geq p$. For the weight function in Example \ref{example: weight 1}, we can simply bound $w(x)\leq (H/h)^{d-p+\beta}$, and thus the weighted inequality leads to
    \[\|u-\frac{1}{h^{d}}\int_D u \|_{L^p(\Omega)} \leq C(d,p)H(\frac{H}{h})^{\frac{d-p}{p}}(\frac{H}{h})^{\frac{\beta}{p}} \|Du\|_{L^p(\Omega)}\, . \]
    This is weaker than the subsampled Poincar\'e inequality up to a polynomial term of $H/h$, because $\beta>0$. In Example \ref{example: weight 2}, we bound $w(x)\leq (H/h)^{d-p}(\log(1/h)+1)^\gamma/\left(\log(\frac{1}{H})+1\right)^{\gamma-p-1}$, and it leads to the rate \[\left(\frac{H}{h}\right)^{\frac{d-p}{p}}\left(\log(\frac{1}{h})+1\right)^{\frac{\gamma}{p}}\] on the small scale parameter $h$, for $\gamma>p-1$. Compared to the rate $\rho_{p,d}(H/h)$ in Proposition \ref{example: subsample Poincare}, the result obtained by Example \ref{example: weight 2} is a little weaker up to a logarithmic term of $H/h$, but it is stronger than Example \ref{example: weight 1}.

    Thus, in terms of deriving the previous subsampled Poincar\'e inequality, the weighted inequality here is not optimal for $p>1$. We will discuss more the optimality of the inequality in the next subsection, regarding the zero-limit of $h$.    
    \subsubsection{Small Limit of Parameter $h$} Recall our motivation for considering the weighted inequality is to tackle the small $h$ issue. The inequality in the last subsection is non-asymptotic in $h$, i.e., it holds when $h$ is a finite number. In this subsection, we take $h$ to $0$ and see what happens for this Poincar\'e inequality.
    
    Let us consider $p>1$, and the weight function is given by Example \ref{example: weight 1}. When $h\to 0$, assume $0\in \Omega$ and $D$ converges to the single point $0$, then the weight function converges to $w(x)=H^{d-p+\beta}|x|^{-(d-p+\beta)}$. The right-hand side of the inequality \eqref{eqn: weighted poincare} converges to 
    \[C(d,p)H^{\frac{d+\beta}{p}}\left(\int_{\Omega} |x|^{-(d-p+\beta)}|Du(x)|^p \right)^{\frac{1}{p}}\, . \]
    For the left-hand side, we need to study whether $\int_D u/h^d$ will attain a limit as $h \to 0$. Indeed, we have the following proposition. The proof technique is similar to that of proving Lemma 2.1 and Theorem 2.2 in \cite{calder_properly-weighted_2018}.
    \begin{proposition}
    \label{prop: u(0) well defined}
    Let $0 \in \Omega \subset \bR^d$ satisfies Assumption \ref{assumption: regular domain}, and $d\geq p>1$. For a function $u \in W^{1,p}(\Omega)$, if 
    \[ \int_{\Omega} |x|^{-(d-p+\beta)}|Du(x)|^p < \infty  \]
    for some $\beta>0$, then we can define \[u(0):=\lim_{h \to 0} \frac{1}{\mu_d(B^d(0,h))}\int_{B^d(0,h)} u\, .\]
    Thus, the pointwise value of $u$ at $x=0$ makes sense, and we have the weighted inequality for $h=0:$
        \begin{equation}
        \label{eqn: pointwise weighted inequality}
           \|u-u(0) \|_{L^p(\Omega)} \leq C(d,p)H \|Du\|_{L^p_w(\Omega)}\, , 
        \end{equation}
        with the weight function given by
        $w(x)=(H/|x|)^{d-p+\beta}$.
    \end{proposition}
    \begin{proof}
    Since $0 \in \Omega$ is an open domain, there exists $h_0>0$ such that for all $0<h<h_0$, $B(x,h) \subset \Omega$. For any $0<h_l<h_k <h_0$, we have
    \begin{equation}
    \label{eqn: small limit, triangle inequ}
        \begin{aligned}
        &|\frac{1}{\mu_d(B^d(0,h_l))}\int_{B^d(0,h_l)}u-\frac{1}{\mu_d(B^d(0,h_k))}\int_{B^d(0,h_k)}u|\\
        \leq &C(d,p)h_k^{-\frac{d}{p}}(\|u-\frac{1}{\mu_d(B^d(0,h_l))}\int_{B^d(0,h_l)}u\|_{L^p(B^d(0,h_k))}\\
        &+\|u-\frac{1}{\mu_d(B^d(0,h_k))}\int_{B^d(0,h_k)}u\|_{L^p(B^d(0,h_k))})
        \end{aligned}
    \end{equation}
    by the triangle inequality and volume calculation. For the first term in the bracket above, the weighted inequality in Theorem \ref{thm: weighted Poincare} and Example \ref{example: weight 1} implies (substitute $\Omega=B^d(0,h_k)$ and $D=B^d(0,h_l)$ here)
    \begin{align*}
        \|u-\frac{1}{B^d(0,h_l)}\int_{B^d(0,h_l)}u\|_{L^p(B^d(0,h_k))}\leq &C(d,p)h_k^{\frac{d+\beta}{p}} \left(\int_{B^d(0,h_k)} |x|^{-(d-p+\beta)}|Du(x)|^p \right)^{\frac{1}{p}}\\
        \leq &C(d,p)h_k^{\frac{d+\beta}{p}} \left(\int_{\Omega} |x|^{-(d-p+\beta)}|Du(x)|^p \right)^{\frac{1}{p}}\, .
    \end{align*}
    This also holds for the second term in the bracket of equation \eqref{eqn: small limit, triangle inequ}. Thus, we get
    \begin{align*}
        &|\frac{1}{\mu_d(B^d(0,h_l))}\int_{B^d(0,h_l)}u-\frac{1}{\mu_d(B^d(0,h_k))}\int_{B^d(0,h_k)}u|\\
        \leq &C(d,p)h_k^{\frac{\beta}{p}} \left(\int_{\Omega} |x|^{-(d-p+\beta)}|Du(x)|^p \right)^{\frac{1}{p}}\, .
    \end{align*}
    By the convergence theorem for Cauchy's series, we obtain \[\lim_{h\to 0} \frac{1}{\mu_d(B^d(0,h))}\int_{B^d(0,h)} u\] exists and we define it to be the pointwise value $u(0)$. Taking $h \to 0$ in Theorem \ref{thm: weighted Poincare} with Example \ref{example: weight 1} leads to the weighted inequality
    \begin{equation*}
           \|u-u(0) \|_{L^p(\Omega)} \leq C(d,p)H \|Du\|_{L^p_w(\Omega)}\, , 
        \end{equation*}
        with the weight function given by
        $w(x)=(H/|x|)^{d-p+\beta}$. The proof is completed.
    \end{proof}
    We provide several remarks below:
    \begin{enumerate}
        \item For the weight function in Example \ref{example: weight 2}, results similar to those stated in Proposition \ref{prop: u(0) well defined} hold by using the same strategy in the proof. These singular weight functions allow the pointwise value of $u$ to be well-defined.
        \item If we set $\beta=0$ in Example \ref{example: weight 1} or $\gamma=p-1$ in Example \ref{example: weight 2}, then the corresponding weighted gradient norm being finite is not enough to guarantee a well-defined pointwise $u(0)$. More precisely, in Example \ref{example: weight 1}, if $\beta=0$, then the assumption on $u$ (taking $h \to 0$) is
        \[\int_{\Omega} |x|^{-(d-p)}|Du(x)|^p < \infty\, . \]
        For $p>1$, the function $u(x)=\log\log (1/|x|+1)$ satisfies this assumption, while $u(0)=\infty$. Thus, the pointwise value $u(0)$ is not a well-defined finite number, and the weighted Poincar\'e inequality for pointwise measurements does not hold. 
        
        For Example \ref{example: weight 2}, when $\gamma=p-1$, the counterexample can be chosen as $u(x)=\log\log\log(1/|x|+1)$. Therefore, the weight functions of Examples \ref{example: weight 1} and \ref{example: weight 2} are ``optimal'' in their family in the sense that a relaxed version cannot guarantee the function to have a well-defined pointwise value, and the corresponding Poincar\'e inequality for pointwise measurement does not hold.
    \end{enumerate}
    One could also propose new weight functions by making $\gamma=p-1$ in Example \ref{example: weight 2} and then adding iterated logarithmic corrections. It is of future interest to see whether this procedure would lead to some sensible limit when we add more and more iterated logarithmic corrections.
    \subsection{Application: Non-Degenerate Recovery}
    In this subsection, we discuss the application of the weighted Poincar\'e inequality for non-degenerate recovery when $d\geq p$. 
    \subsubsection{Domain and Decomposition}
    For simplicity, we consider the same domain as in the introduction: $\Omega= [0,1]^d=\bigcup_{i\in I} \omega_i^H$. The subsampled domain $\omega^{h,H}_i \subset \omega_i^H$; see Figure \ref{fig:two-scalegrid}. We denote $\Omega^{h,H}:=\bigcup_{i\in I} \omega_i^{h,H}$ which is the region that the subsampled data depend on. The recovery problem is to recover function $u$ after seeing the data $[u,\phi_i^{h,H}]$ for every $i \in I$; the notations $[\cdot,\cdot]$ and $\phi^{h,H}_i$ are the same as those defined in Section \ref{sec: improving the basis function}. 
    
    For each local patch $\omega_i^{H}, i \in I$, its center is denoted by $x_i \in \omega_i^{H}$. We write $X^H=\bigcup_{i=1}^I \{x_i^H \}$. We assume there is a sequence of subsampled domain $\omega_i^{h,H}$ indexed by $h$, and for each $i \in I$, $\bigcap_{h>0}\omega^{h,H}_i=\{x_i\}$. This assumption is natural when we want to study the degeneracy issue in approximation, i.e., eventually we will let the small scale parameter $h$ goes to zero; in the limit, the data we have becomes $\{u(x_i)\}_{i \in I}$.
	\subsubsection{Weight Function} We adopt the profile of weight function in Example \ref{example: weight 1}. Let $$w^{h,H}(x):=\left(\frac{H}{\max \{h,\sfd(x,\Omega^{h,H})\}}\right)^{d-p+\beta}$$
	for some $\beta >0$. When $h$ approaches $0$, it converges to
	\[w^{H}(x)=\left(\frac{H}{\sfd(x,X^H) }\right)^{d-p+\beta} \, .\]
	\subsubsection{Piecewise Constant Recovery}
	Based on the data $\{[u,\phi_i^{h,H}]\}_{i=1}^I$, we can estimate the error of the piecewise constant recovery using the weighted Poincar\'e inequality. In each local patch $\omega^{H}_i$, we have the error bounded by \[C(d,p,\beta) H\left(\int_{\omega^{H}_i} w^{h,H}(x)|Du(x)|^p \, \rd x\right)^{\frac{1}{p}}\, ,\] according to Theorem \ref{thm: weighted Poincare} and Example \ref{example: weight 1}. Here $C(d,p,\beta)$ is a constant that depends on $d,p, \beta$ only. Then, summing all the errors for $i \in I$, we get the overall error in the domain $\Omega$ upper bounded by
	\[C(d,p,\beta) H \left(\int_{\Omega} w^{h,H}(x)|Du(x)|^p \, \rd x\right)^{\frac{1}{p}} \, . \]
	We can get a universal upper bound of the above term by letting $h\to 0$ due to the monotonicity:
	\[C(d,p,\beta) H \left(\int_{\Omega} w^{H}(x)|Du(x)|^p \, \rd x\right)^{\frac{1}{p}}\, .\]
	Hence, if we have the assumption that $\int_{\Omega} w^{H}(x)|Du(x)|^p \, \rd x < \infty$, then the error estimate would not degrade as $h$ goes to $0$; the upper bound can be independent of $h$. The formula also tells us that acquiring data at these singular locations will be of the first importance if we aim to recover $u$, because $Du$ is small and $u$ is nearly flat around these regions.
	\subsubsection{Improved Basis Functions}
	To construct the improved basis function for $p=2$, we follow the same step in Section \ref{sec: improving the basis function}. We treat the weight function $w^{h,H}(x) \in L^{\infty}(\Omega)$ as the role of $a$ when $h$ is finite. Define the basis function by:
	\begin{equation*}
	\psi_{i}^{h,H,w} = \text{argmin}_{\psi \in H_0^1(\Omega)}\,  \int_{\Omega} w^{h,H}(x)|D\psi(x)|^2 \quad \text{s.t.} \quad 
	\int_{\omega_i^{H}} \psi \phi_j^{h,H} = \delta_{i,j} \quad \text{for} \quad  j \in I\, .
	\end{equation*}
	As before, the recovered solution is constructed by
	\[u^{h,H,w}=\sum_{i=1}^I [u,\phi_{i}^{h,H}]\psi_{i}^{h,H,w}\, . \]
	We have the following error estimate of the recovery. It is non-asymptotic regarding $h$.
	\begin{theorem}
		\label{thm:approximation error, weighted spline}
		Under the assumption that $u \in H_0^1(\Omega)$, we have the following error estimate:
        \begin{align*}
        & \int_{\Omega} w^{h,H}|D(u-u^{h,H,w})|^2 \leq \int_{\Omega} w^{h,H}|Du|^2 \\
        & \int_{\Omega} |u-u^{h,H,w}|^2 \leq C(d)H^2\int_{\Omega} w^{h,H}|Du|^2\, ,
        \end{align*}
        where $C(d)$ is a constant that depends on $d$ only.
        
        Furthermore, under the assumption that $f^h:=\nabla \cdot (w^{h,H} \nabla u) \in L^2(\Omega)$, we have the improved energy estimate:
        \[ \int_{\Omega} w^{h,H}|D(u-u^{h,H,w})|^2 \leq C(d)H^2\|f^h\|^2_{L^2(\Omega)} \, , \]
        and the improved $L^2(\Omega)$ estimate:
        \[\int_{\Omega} |u-u^{h,H,w}|^2 \leq C(d)H^4\|f^h\|^2_{L^2(\Omega)}\, .\]
        The constant $C(d)$ can vary from place to place.
	\end{theorem}
	\begin{proof}
	 Substituting $a=w^{h,H}$ in Theorem \ref{thm:approximation error, subsampled spline} concludes the proof.
	\end{proof}
	Let us discuss the implication of this theorem. If we have the regularity assumption on $u$ that $\int_{\Omega} w^{H}|Du|^2 < \infty$, then the $L^2$ error will be bounded by \[C(d)H(\int_{\Omega} w^{H}(x)|Du(x)|^2\, \rd x)^{1/2}\] for any $h > 0$. Thus, under this regularity assumption, the estimate in Theorem \ref{thm:approximation error, weighted spline} survives in the zero-limit of $h$, while the estimate in Theorem \ref{thm:approximation error, subsampled spline} blows up. Therefore, the weighted inequality is needed to study the small $h$ regime.
    
    Furthermore, if we know additionally that $\sup_{h > 0}\|f^h\|_{L^2(\Omega)} < \infty$, then the error in the energy norm will be bounded by $C(d)H\sup_{h > 0}\|f^h\|_{L^2(\Omega)}$, while the $L^2$ error is bounded by $C(d)H^2\sup_{h > 0}\|f^h\|_{L^2(\Omega)}$ for any $h > 0$. The rate is better than before, and no blow-up occurs in the small $h$ limit. It is of future interest to look at in which practical scenario the assumption $\sup_{h > 0}\|f^h\|_{L^2(\Omega)} < \infty$ is possible to hold.
    
    We remark that other weight functions such as Example \ref{example: weight 2} can also be used in this subsection; the results are similar. The key is the weighted Poincar\'e inequality holds and the pointwise value is well-defined in the small $h$ limit.
    \section{Concluding Remarks}
    \label{sec: discussion}
    In this paper, we have studied the subsampled Poincar\' e inequality as a tool for function approximation and recovery. The context of subsampled data introduces an additional scale parameter (which is $h$) into the problem. It is important to capture the dependence of the approximation accuracy on $h$. For this purpose, we have developed some analytic tools that can be used to analyze the recovery error in the finite $h$ and zero-limit regime.
    
    In the finite $h$ regime, we demonstrated the optimality of the subsampled Poincar\'e inequality concerning the parameter $h$. The sliced data case was also investigated. We proved that the corresponding Poincar\'e inequality is optimal in the case $d\neq p$ and nearly optimal up to a logarithmic term in the critical case $d=p$. It is of future interest to improve the rate to the optimum in the critical case, and to generalize the results in the paper to function space beyond $W^{1,p}(\Omega)$. 
    
    When $d\geq p$, the error estimates obtained by the subsampled Poincar\'e inequality blows up as $h \to 0$; thus, it fails in the small $h$ limit. To identify a sensible limit, we assumed the function $u$ belongs to a weighted space, and developed a weighted Poincar\'e inequality to analyze the recovery error. The weighted estimates remain valid in the small $h$ limit, leading to non-degenerate function recovery in the zero-limit regime. 
    
    We note that our discussion on the weighted Poincar\'e inequality connects to the discussion on the degeneracy issue in graph Laplacian based semi-supervised learning approaches \cite{Semi-supervised}, which are formulated as discrete function recovery problems. Adjusting the weights of the Laplacian to achieve desired recovery performance is essential in practice. Recently, the authors in \cite{calder_properly-weighted_2018} established the consistency of the properly weighted graph Laplacian approach. The weight function there has the same form as our Example \ref{example: weight 1}. These Sobolev critical functions help regularize the process to obtain a non-degenerate recovery in the small data regime.
    
    Another main topic in this paper is the choice of basis functions. The direct use of Poincar\'e's inequality corresponds to piecewise constant basis functions, which achieve the same error rate as the Poincar\'e inequality indicates. Further, based on ideas from the spline approximation theory, we can improve the regularity of the basis function by solving some variational problems. These improved basis functions enhance recovery accuracy when the underlying function has better regularity. 
    
    As noted in the introduction, it is possible to use these basis functions for solving multiscale PDE problems, as in \cite{owhadi_multigrid_2017, owhadi2019operator}. We will discuss this topic in our companion numerical paper \cite{chen-hou_numerical}, regarding the tradeoff between the subsampled scale $h$, the exponential decay rate of the basis function, and the accuracy of the approximate solution.

    \vspace{0.2in}
    \noindent
    {\bf Acknowledgments}. The research was in part supported by NSF Grants DMS-1912654 and DMS-1907977. Y. Chen is supported by the Kortschak Scholars Program. We want to thank Professor Henri Berestycki and Jinchao Xu for their interest in our work and for bringing to our attention some of the relevant references. Y. Chen would like to thank Yousuf Soliman for many insightful discussions on the subsampled Poincar\'e inequality. We thank the anonymous reviewer for the helpful comments that improve this work.

\section{Appendix}
\label{sec: appendix}
\subsection{Proof of Proposition \ref{example: subsample Poincare}}
\label{subsec: proof subsampled poincare}
    \begin{proof}[Proof of Proposition \ref{example: subsample Poincare}]
        Let the measure $\lambda$ in Theorem \ref{thm: general Poincare} be supported on $D$ and uniform in $D$. Then, $\frac{1}{h^d}\int_D u=\int_{\overline{\Omega}} u \, \rd\lambda$. Hence, we have
        \[\|u-\frac{1}{h^{d}}\int_D u \|_{L^p(\Omega)} \leq \text{diam}(\Omega)\left(\int_0^1 \frac{\alpha(t)^{\frac{1}{p}}}{t^{\frac{d}{p}}}\, \rd t \right)\|Du\|_{L^p(\Omega)}\, ,  \]
        where $\alpha(t)$ is an upper bound on $\lambda(\frac{z-t\Omega}{1-t}\cap \Omega) $. A trivial bound is $\alpha(t)\leq 1$. On the other hand, since $\lambda$ is supported on $D$, we have $$\lambda(\frac{z-t\Omega}{1-t}\cap \Omega)=\lambda(\frac{z-t\Omega}{1-t}\cap D) \leq \frac{1}{h^d}\mu_d(\frac{z-t\Omega}{1-t})\leq \frac{H^d}{h^d}(\frac{t}{1-t})^d,$$
        where we have used the fact that the density of $\lambda$ on $D$ is $\frac{1}{h^d}$. Thus, we choose 
        \[\alpha(t)=\min \{1, \frac{H^d}{h^d}(\frac{t}{1-t})^d \} = \frac{H^d}{h^d}(\frac{t}{1-t})^d\cdot \chi_{[0,\frac{h}{H+h})}(t)+1\cdot\chi_{[\frac{h}{H+h},1]}(t) \, . \]
        We then calculate the integral:
        \begin{equation}
        \label{eqn: coef of Poincare}
        \begin{aligned}
        \int_0^1 \frac{\alpha(t)^{\frac{1}{p}}}{t^{\frac{d}{p}}}\, \rd t=(\frac{H}{h})^{\frac{d}{p}}\int_0^{\frac{h}{H+h}} \frac{1}{(1-t)^{\frac{d}{p}}} \, \rd t + \int_{\frac{h}{H+h}}^1 \frac{1}{t^{\frac{d}{p}}} \, \rd t\, .
        \end{aligned}
        \end{equation}
        When $d<p$, the integral in \eqref{eqn: coef of Poincare} becomes 
        \begin{equation}
        \label{eqn: integral term 1}
        \begin{aligned}
        \frac{p}{p-d}\left((\frac{H}{h})^{\frac{d}{p}}(1-(\frac{H}{H+h})^{1-\frac{d}{p}})+1-(\frac{h}{H+h})^{1-\frac{d}{p}} \right) \, .
        \end{aligned}
        \end{equation}
        Since $-1<\frac{d}{p}-1 < 0$, by Bernoulli's inequality, we have
        \[(\frac{H}{h})^{\frac{d}{p}}(1-(\frac{H}{H+h})^{1-\frac{d}{p}})=(\frac{H}{h})^{\frac{d}{p}}(1-(1+\frac{h}{H})^{\frac{d}{p}-1})\leq (\frac{H}{h})^{\frac{d}{p}}\frac{h}{H}(1-\frac{d}{p}) \leq 1-\frac{d}{p}\, ,  \]
        where we have used the fact $(\frac{H}{h})^{\frac{d}{p}}\frac{h}{H}=(\frac{h}{H})^{1-\frac{d}{p}}\leq 1$. Thus, we have the quantity in \eqref{eqn: integral term 1} bounded by $$\frac{p}{p-d}(1-\frac{d}{p}+1)=\frac{2p-d}{p-d} \leq C(d,p) \, .$$
        When $d=p$, the integral in \eqref{eqn: coef of Poincare} is
        \begin{equation}
        \label{eqn: integral term 2}
        \begin{aligned}
        \frac{H}{h}\ln(1+\frac{h}{H})+\ln(1+\frac{H}{h})\leq 1+\ln(1+\frac{H}{h}) \leq C\ln (1+\frac{H}{h}) \, .
        \end{aligned}
        \end{equation}
        When $d > p$, the integral in \eqref{eqn: coef of Poincare} becomes 
        \begin{equation}
        \label{eqn: integral term 3}
        \begin{aligned}
        \frac{p}{d-p}\left((\frac{H}{h})^{\frac{d}{p}}((1+\frac{h}{H})^{\frac{d-p}{p}}-1)+(1+\frac{H}{h})^{\frac{d-p}{p}}-1 \right) \leq C(d,p) (\frac{H}{h})^{\frac{d-p}{p}} \, .
        \end{aligned}
        \end{equation}
        The proof is completed.
    \end{proof}
    \subsection{Proof of Proposition \ref{prop: sharpness of the rate}}
    \label{subsec: sharpness subsampled}
\begin{proof}[Proof of Proposition \ref{prop: sharpness of the rate}]
        We construct the sequence $u_h$ explicitly. For $d=p$, we take 
        \[u_h(x)=\frac{\max~\{0, \ln (1+\frac{|x|}{h})-\ln 2\} }{\ln(1+\frac{1}{h})}\, . \]
        Then $u_h(x)$ equals $0$ in $D_h$. Thus, we have
        \begin{align*}
        \|u_h-\frac{1}{\mu_d(D_h)}\int_{D_h}u_h\|_{L^p(\Omega)}^p&=\int_{B^d(0,1)\backslash B^d(0,h)} u_h^p\\
        &=\mu_{d-1}(\bS^d)\int_h^1 \frac{\max~\{0, \ln (1+\frac{r}{h})-\ln 2\}^p }{\ln(1+\frac{1}{h})^p} r^{d-1}\, \rd r\\
        &\geq \mu_{d-1}(\bS^d) \frac{(\ln (1+\frac{1}{2h})-\ln 2)^p }{\ln(1+\frac{1}{h})^p} \int_{1/2}^1  r^{d-1}\, \rd r\\
        &\geq C(d,p)
        \end{align*}
        for some $C(d,p)>0 $ independent of $h$; we have used the condition $h\leq 1/2$  and the fact $\lim_{h \to 0} \frac{\ln (1+\frac{1}{2h})}{\ln (1+\frac{1}{h})}=1$. Here we use $\bS^d$ to represent the $d$ dimensional unit sphere. On the other hand, we obtain
        \begin{align*}
        \|Du_h\|_{L^p(\Omega)}^p&=\frac{1}{(\ln (1+\frac{1}{h}))^p}\int_{B^d(0,1)\backslash B^d(0,h)} \frac{1}{(h+|x|)^p}\, \rd x\\
        &=\mu_{d-1}(\bS^d)\frac{1}{(\ln (1+\frac{1}{h}))^p}\int_h^1 \frac{r^{d-1}}{(h+r)^p}\, \rd r\\
        &\leq C(d,p)\frac{1}{(\ln (1+\frac{1}{h}))^{d-1}}
        \end{align*}
        for some $C(d,p)$ dependent of $d,p$. In the last step, we have used the inequality $h+r \geq r$ and the fact that $\lim_{h \to 0} \frac{\ln (1+\frac{1}{h})}{\ln (\frac{1}{h})}=1$.
        
        Hence, for this sequence $u_h$, we get
        \[\frac{\|u_h-\frac{1}{\mu_d(D_h)}\int_{D_h}u_h\|_{L^p(\Omega)}}{\|Du_h\|_{L^p(\Omega)}}\geq C(d,p)(\ln(1+\frac{1}{h}))^{\frac{d-1}{d}}=C(d,p)\rho_{p,d}(\frac{1}{h})\, . \]
        For $d > p$, we construct 
        \[u_h(x)=\min~\{\frac{\max~\{|x|-h,0\}}{h},1\} \, . \]
        Then, $u_h(x)$ vanishes in $D_h$, and 
        \begin{align*}
        \|u_h-\frac{1}{\mu_d(D_h)}\int_{D_h}u_h\|_{L^p(\Omega)}^p&=\int_{B^d(0,1)\backslash B^d(0,h)} u_h^p\\
        &\geq \int_{B^d(0,1)\backslash B^d(0,1/2)} u_h^p = C(d,p)\, ,
        \end{align*}
        where $C(d,p)$ is independent of $h$. Here we have used the fact $h\leq 1/4$ and $u_h=1$ when $|x|\geq 1/2$. In the meanwhile, we get
        \begin{align*}
        \|Du_h\|_{L^p(\Omega)}^p&=\int_{B^d(0,2h)\backslash B^d(0,h)} \frac{1}{h^p}\, \rd x = C(d,p)h^{d-p}\, .
        \end{align*}
        Hence, we conclude that
        \[\frac{\|u_h-\frac{1}{\mu_d(D_h)}\int_{D_h}u_h\|_{L^p(\Omega)}}{\|Du_h\|_{L^p(\Omega)}}\geq C(d,p)h^{\frac{p-d}{p}}=C(d,p)\rho_{p,d}(\frac{1}{h})\, . \]
        The proof is completed.
    \end{proof}
    \subsection{Proof of Proposition \ref{example: sliced Poincare}}
    \label{subsec: proof sliced poincare}
    \begin{proof}[Proof of Proposition \ref{example: sliced Poincare}]
        Similar to the proof of Proposition \ref{example: subsample Poincare}, we first characterize $\alpha(t)$, and then calculate the related integral. Since $\lambda$ is supported on $\Gamma$, we have $$\lambda(\frac{z-t\Omega}{1-t}\cap \Omega)=\lambda(\frac{z-t\Omega}{1-t}\cap \Gamma) \leq \frac{H^{d-1}}{h^{d-1}}(\frac{t}{1-t})^{d-1},$$
        where we have used the fact that the density of $\lambda$ on the $d-1$ dimensional $\Gamma$ is $\frac{1}{h^{d-1}}$. Hence, we choose
        \[\alpha(t)=\min \{1, \frac{H^{d-1}}{h^{d-1}}(\frac{t}{1-t})^{d-1} \} = \frac{H^{d-1}}{h^{d-1}}(\frac{t}{1-t})^{d-1}\cdot \chi_{[0,\frac{h}{H+h})}(t)+1\cdot\chi_{[\frac{h}{H+h},1]}(t) \, . \]
        The corresponding integral is
        \begin{equation}
        \label{eqn: coef of Poincare, sliced}
        \begin{aligned}
        \int_0^1 \frac{\alpha(t)^{\frac{1}{p}}}{t^{\frac{d}{p}}}\, \rd t=(\frac{H}{h})^{\frac{d-1}{p}}\int_0^{\frac{h}{H+h}} \frac{1}{t^{\frac{1}{p}}(1-t)^{\frac{d-1}{p}}} \, \rd t + \int_{\frac{h}{H+h}}^1 \frac{1}{t^{\frac{d}{p}}} \, \rd t\, .
        \end{aligned}
        \end{equation}
        For the first term in \eqref{eqn: coef of Poincare, sliced}, 
        \begin{align*}
        (\frac{H}{h})^{\frac{d-1}{p}}\int_0^{\frac{h}{H+h}} \frac{1}{t^{\frac{1}{p}}(1-t)^{\frac{d-1}{p}}} \, \rd t \leq & (\frac{H}{h})^{\frac{d-1}{p}} (1+(1+\frac{h}{H})^{\frac{d-1}{p}}) \int_0^{\frac{h}{H+h}} \frac{1}{t^{\frac{1}{p}}} \, \rd t\\
        = & \frac{p}{p-1}(\frac{H}{h})^{\frac{d-1}{p}} (1+(1+\frac{h}{H})^{\frac{d-1}{p}}) (1+\frac{H}{h})^{\frac{1}{p}-1} \\
        = & \frac{H^{\frac{d-1}{p}}(H+h)^{\frac{1}{p}-1}}{h^{\frac{d-p}{p}}} (1+(1+\frac{h}{H})^{\frac{d-1}{p}})\\
        \leq & (2^{\frac{1}{p}-1}+1)(\frac{H}{h})^{\frac{d-p}{p}}(2+2^{\frac{d-1}{p}}) \, ,
        \end{align*} 
        where in the last step we have used the estimate \[(H+h)^{\frac{1}{p}-1}\leq H^{\frac{1}{p}-1}+(2H)^{\frac{1}{p}-1} \quad \text{and} \quad (1+\frac{h}{H})^{\frac{d-1}{p}}\leq 1 + 2^{\frac{d-1}{p}}\, .\] This is due to $0 \leq h \leq H$ and the fact that, the value of an one dimensional non-negative monotone function will not be larger than the sum of its two endpoint values in an interval. Observe that the last term in the above calculation will be bounded by a constant $C(d,p)$ if $d \leq p$ and by $C(d,p)(H/h)^{\frac{d-p}{p}}$ if $d > p$. Moreover, the second term in \eqref{eqn: coef of Poincare, sliced} is the same as in \eqref{eqn: coef of Poincare}. Thus, the same argument there can be applied here. Finally, we obtain the Poincar\'e inequality with the same $\tilde{\rho}_{p,d}$ dependence on $H/h$ as Proposition \ref{example: subsample Poincare}. 
		\end{proof}
\subsection{Proof of Proposition \ref{prop: optimal rate demon, sliced case}}
\begin{proof}[Proof of Proposition \ref{prop: optimal rate demon, sliced case}]
    The proof is the same as the proof of Proposition \ref{prop: sharpness of the rate}. The critical examples that achieve the lower bound are the same.    
\end{proof}
	
\subsection{Proof of Theorem \ref{thm: weight Poincare W11}}
\label{Proof of Theorem thm: weight Poincare W11}
\begin{proof}[Proof of Theorem \ref{thm: weight Poincare W11}]
		Assumption \ref{assumption: regular domain} implies $C_1\text{diam}(\Omega) \leq H \leq C_2\text{diam}(\Omega)$ and $C_1\text{diam}(D) \leq h \leq C_2\text{diam}(D)$. We use the result in our Theorem \ref{thm: general Poincare} to get:
		\begin{align}
		\label{eqn: in weight Poincare}
		\|u-\int_{\Omega} u \,\rd\lambda \|_{L^1(\Omega)}&\leq \text{diam}(\Omega)\int_{\Omega} \left(\int_0^1 \frac{1}{t^d} \lambda(\frac{z-t\Omega}{1-t}\cap D)\, \rd t \right) |Du(z)|  \, \rd z \, ,
		\end{align}
		where $\lambda=\frac{1}{h^d}\mu_d$ in $D$.
		Now, we characterize $\lambda(\frac{z-t\Omega}{1-t}\cap D)$ in more details, rather than just using a uniform bound $\alpha(t)$ as before. We study when the intersection $\frac{z-t\Omega}{1-t}\cap D$ becomes empty, i.e. $\sfd(\frac{z-t\Omega}{1-t}, D) > 0$. Without loss of generality we assume $0 \in D$, otherwise we can shift the domain to contain the origin. Then, $0 \in D \cap \frac{t\Omega}{1-t}$. If $|z|$ is large then $\frac{z-t\Omega}{1-t}$ will be separated from $D$. A sufficient condition will be
		\[\frac{|z|}{1-t} > \text{diam}(\frac{t\Omega}{1-t})+\text{diam}(D) \geq \frac{1}{C_2}(\frac{tH}{1-t}+h) \, . \]
		This is equivalent to $t \leq \frac{C_2|z|-h}{H-h}$. Thus we obtain
		\begin{align}
		\label{eqn: proof in weighted 0 condition}
		t \leq \frac{C_2|z|-h}{H-h}\quad \Rightarrow \quad \lambda(\frac{z-t\Omega}{1-t}\cap D)=0 \, . 
		\end{align}
		We decompose the integral on the right-hand side of equation \eqref{eqn: in weight Poincare} into two parts (the integrand is abbreviated as $I$): 
		\[\int_{\Omega} I \rd z = \int_{\{C_2|z| < 2h\}\cap \Omega} I\rd z +\int_{\{C_2|z|\geq 2h\} \cap \Omega} I\rd z \, .\]
		For the first part, we use the result in Corollary \ref{example: subsample Poincare}:
		\begin{equation}
		\begin{aligned}
		\int_{\{C_2|z| < 2h\}\cap \Omega} I\rd z &\leq C(d,p)(\frac{H}{h})^{d-1}\int_{\{C_2|z| < 2h\}\cap \Omega} |Du(z)| \, \rd z \\
		&\leq C(d,p)\int_{\{C_2|z| < 2h\}\cap \Omega} \left(\frac{H}{\max \{h,|z|\}}\right)^{d-1}|Du(z)| \, \rd z\\
		&\leq C(d,p)\int_{\{C_2|z| < 2h\}\cap \Omega} w(z)|Du(z)| \, \rd z
		\end{aligned}
		\end{equation}
		where the last line is due to $\sfd(z,D) \leq |z|$.\\
		For the second part, we have $C_2|z| \geq 2h$. Due to equation \eqref{eqn: proof in weighted 0 condition}, for $z \in \{C_2|z|\geq 2h\} \cap \Omega$ and at the same time $z \in \{C_2|z| \leq H\}$, we have
		\begin{equation}
		\begin{aligned}
		\int_0^1 \frac{1}{t^d} \mu_d(\frac{z-t\Omega}{1-t}\cap D)\, \rd t&\leq \int_{\frac{C_2|z|-h}{H-h}}^1 \frac{1}{t^d} \, \rd t\\
		&\leq \frac{1}{d-1} \left( (\frac{C_2|z|-h}{H-h})^{1-d}-1 \right)\\
		&\leq C(d,p)(\frac{H}{|z|})^{d-1} \leq C(d,p) w(z) \, ,
		\end{aligned}
		\end{equation}
		where the last two lines are due to the relation $0 \leq h \leq \frac{C_2}{2}|z|$ and $\sfd(z,D) \leq |z|$.
		For $z \in \{C_2|z|\geq 2h\} \cap \Omega$ and also $z \in \{C_2|z| > H\}$, the integral vanishes due to equation \eqref{eqn: proof in weighted 0 condition}. Combining all these together, we arrive at
		\[\|u-\frac{1}{h^{d}}\int_D u \|_{L^1(\Omega)} \leq C(d,p)H \|Du\|_{L^1_w(\Omega)} \, . \]
		This completes the proof.
	\end{proof}
	\subsection{Proof of Theorem \ref{thm: weighted Poincare}}
	\label{Proof of Theorem thm: weighted Poincare}
	\begin{proof}[Proof of Theorem \ref{thm: weighted Poincare}] By the triangle inequality, we get
		\begin{equation}
		\label{eqn: proof in weighted Poincare 1}
		\begin{aligned}
		\|u-\frac{1}{h^{d}}\int_D u \|_{L^p(\Omega)} &\leq \|u-\frac{1}{H^{d}}\int_\Omega u\|_{L^p(\Omega)}+H^{\frac{d}{p}}|\frac{1}{H^{d}}\int_\Omega u-\frac{1}{h^{d}}\int_D u|\\
		&\leq C(d,p)H\|Du\|_{L^p(\Omega)}+H^{\frac{d}{p}-d}\int_{\Omega}\int_D \frac{1}{h^d}|u(x)-u(y)| \, \rd x \rd y \, ,
		\end{aligned}
		\end{equation}
		where we have used the standard Poincar\' e inequality for the first part. For the second  part, due to the proof in Theorem \ref{thm: general Poincare} and Theorem \ref{thm: weight Poincare W11}, we have 
		\begin{equation*}
		\begin{aligned}
		H^{\frac{d}{p}-d}\int_{\Omega}\int_D &\frac{1}{h^d}|u(x)-u(y)| \, \rd x \rd y\\ 
		&\leq C(d,p)H^{\frac{d}{p}-d+1}\int_{\Omega} \left(\frac{H}{\max \{h,\sfd(z,D)\}}\right)^{d-1}|Du(z)|\, \rd z \, .
		\end{aligned}
		\end{equation*}
		Using the H\"older inequality, we get
		\begin{equation*}
		\begin{aligned}
		&\int_{\Omega} \left(\frac{H}{\max \{h,\sfd(z,D)\}}\right)^{d-1}|Du(z)|\, \rd z\\
		=&\int_{\Omega} \left(\frac{H}{\max \{h,\sfd(z,D)\}}\right)^{d-1}w(z)^{-\frac{1}{p}}\cdot w(z)^{\frac{1}{p}}|Du(z)|\, \rd z\\
		\leq & \left(\int_{\Omega} \left(\frac{H}{\max \{h,\sfd(z,D)\}}\right)^{\frac{p(d-1)}{p-1}}w(z)^{-\frac{1}{p-1}}\, \rd z \right)^{\frac{p-1}{p}}\|Du\|_{L^p_w(\Omega)} \\
		\leq & C_w^{1-\frac{1}{p}}H^{d-\frac{d}{p}} \|Du\|_{L^p_w(\Omega)} \, .
		\end{aligned}
		\end{equation*}
		Plugging this into equation \eqref{eqn: proof in weighted Poincare 1} gives
		\begin{align*}
		    \|u-\frac{1}{h^{d}}\int_D u \|_{L^p(\Omega)} &\leq C(d,p)H\|Du\|_{L^p(\Omega)}+C_w^{1-\frac{1}{p}}C(d,p) H\|Du\|_{L^p_w(\Omega)} \\
		    &\leq C(d,p) H \|Du\|_{L^p_w(\Omega)}
		\end{align*}
		where $C(d,p)$ represents a generic constant that depends on $d$ and $p$ only.
	\end{proof}
\end{document}